\documentclass[reqno,twoside,11pt]{amsart}
\usepackage{amsmath, amsfonts, amssymb, esint, amsthm, nicefrac, bbm, dsfont}
\usepackage{epsfig, graphicx,xcolor}
\newcommand{\dx}{~\mathrm{d}x}
\newcommand{\sym}{\mathrm{sym}}
\newcommand{\curl}{\mathrm{curl}}
\newcommand{\Id}{\mathrm{Id}}

\newcommand{\R}{\mathbb{R}}
\def\endproof{\hspace*{\fill}\mbox{\ \rule{.1in}{.1in}}\medskip }
\newcommand*{\dbar}[1]{\bar{\bar{#1}}}
\newcommand*{\tbar}[1]{\bar{\dbar{#1}}}

\numberwithin{equation}{section}
\theoremstyle{plain}

\newtheorem*{theorem*}{Theorem}
\newtheorem{theorem}{Theorem}[section]
\newtheorem{lemma}[theorem]{Lemma}
\newtheorem{corollary}[theorem]{Corollary}
\newtheorem{proposition}[theorem]{Proposition}

\theoremstyle{definition}

\begin{document}
\title [The Monge-Amp\`ere system in dimension and codimension two]
{The Monge-Amp\`ere system in dimension two is fully flexible in
  codimension two}
\author{Dominik Inauen and Marta Lewicka}
\address{D.I.: Institut f\"ur Mathematik, Universit\"at Leipzig,
  D-04109, Leipzig, Germany}
\address{M.L.: University of Pittsburgh, Department of Mathematics, 
139 University Place, Pittsburgh, PA 15260}
\email{dominik.inauen@math.uni-leipzig.de, lewicka@pitt.edu} 

\date{\today}
\thanks{M.L. was partially supported by NSF grant DMS-2407293.
}

\begin{abstract}
We prove that every $\mathcal{C}^1(\bar\omega)$-regular subsolution of the Monge-Amp\`ere system posed on a $2$-dimensional domain $\omega$ and with target codimension $2$, can be uniformly approximated by its exact solutions with regularity $\mathcal{C}^{1,\alpha}(\bar\omega)$ for any $\alpha<1$. This result asserts flexibility of Poznyak's theorem for isometric immersions of $2$d Riemannian manifolds into $\mathbb{R}^4$, in the parallel setting of the Monge-Amp\`ere system.
\end{abstract}

\maketitle
\tableofcontents

\section{Introduction}

The purpose of this paper is to prove the full flexibility of weak solutions to the Monge-Amp\`ere system in dimension $d=2$ and codimension $k=2$.
The Monge-Amp\`ere system was introduced in \cite{lew_conv} as a higher-dimensional version of the classical Monge-Amp\`ere equation posed on a $2$-dimensional domain, in relation to its interpretation as the prescription, to the leading order terms, of the Gaussian curvature of a shallow surface given as the graph of $v$:
\begin{equation}\label{e:MA}
\begin{split}
& \det \nabla^2 v =f \quad \text{ in } \; \omega \subset \R^2,\\
& \, \mbox{for }\; v:\omega\to\R.
\end{split}
\end{equation}
In the same vein, prescription of the full Riemann curvature tensor $[R_{ij, st}]_{i,j,s,t:1\ldots d}$ of a $d$-dimensional shallow manifold that is the graph of $v:(\omega\subset\R^d)\to\R^k$,  is reflected by:
\begin{equation}\label{e:MAdk}
\begin{split}
& \mathfrak{Det}\, \nabla^2 v \doteq \big[\langle\partial^2_{is}v,\partial^2_{jt}v\rangle - 
\langle\partial^2_{it}v,\partial^2_{js}v\rangle\big]_{i,j,s,t:1\ldots d}
= F \quad \text{ in } \; \omega \subset \R^d,\\
& \mbox{for }\; v:\omega\to\R^k,
\end{split}
\end{equation}
where $F:\omega\to\R^{d^4}$ is a given field. 
Indeed, the Riemann curvatures of the family of embeddings $u_\epsilon :\omega\to\R^{d+k}$ parametrized by $\epsilon\to 0$ and given by $u_\epsilon(x) = (x, \epsilon v(x))$, are calculated as:
$$R_{ij,st}\big((\nabla u_\epsilon)^T\nabla u_\epsilon)\big) =
R_{ij,st}\big(\mbox{Id}_d + \epsilon^2(\nabla v)^T \nabla v\big)= \epsilon^2(\mathfrak{Det}\,\nabla^2v)_{ij,st} + o(\epsilon^2),$$
similarly to the, perhaps more familiar, formula for Gauss's curvature $\kappa$ in case $d=2$, $k=1$:
$$\kappa((\nabla u_\epsilon)^T\nabla u_\epsilon)=\kappa(\mbox{Id}_2+\epsilon^2 \nabla v\otimes\nabla v)= \frac{\epsilon^2\det \nabla^2v}{(1+\epsilon^2|\nabla v|^2)^2} =\epsilon^2\det\nabla^2 v + o(\epsilon^2).$$
Relying on the special structure of $\mathfrak{Det}\,\nabla^2$, we call $v\in H^1(\omega,\R^k)$ a {\em weak solution} to (\ref{e:MAdk}), if the following identity holds in the sense of distributions:
\begin{equation}\label{e:MA_weak}
-\frac{1}{2}\mathfrak{C}^2((\nabla v)^T\nabla v) = F \quad \text{ in } \; \omega.
\end{equation}
Here, $\mathfrak{C}^2$ is a second-order operator introduced in \cite{lew_conv}, which in dimension $d=2$ reduces, up to symmetries, to taking {\em $\mbox{curl}\,\mbox{curl}$} of a $\R^{2\times 2}_{\sym}$- valued matrix field, whereas and for $k=1$, the identity (\ref{e:MA_weak}) becomes exactly the weak formulation of (\ref{e:MA}) as studied in \cite{lewpak_MA}:
$$-\frac12\,\curl\,\curl\left( \nabla v\otimes \nabla v\right) = f.$$
The operator in the left hand side above appeared in \cite{Iwaniec}, called therein the \emph{very weak Hessian}. It can be observed that the latter equation on a simply connected domain $\omega\subset \R^2$, or, more generally, the system (\ref{e:MA_weak}) when posed on a contractible domain $\omega\subset\R^d$, reduces to solving:
\begin{equation}\label{e:VK}
\begin{split}
& \frac12 (\nabla v)^T \nabla v + \sym\, \nabla w = A
\\ & \mbox{for }\; v:\omega\to\R^k, \quad w:\omega\to\R^d,
\end{split}
\end{equation} 
where the matrix field $A:\omega\to \R^{d\times d}_\sym$ satisfies the compatibility condition: $-\mathfrak{C}^2(A) = F$, further equivalent to the satisfaction by $F$ of the appropriate symmetry, algebraic and differential identities. The system (\ref{e:VK}) is called the {\em von K\'arman system}, in relation to the von K\'arman stretching content in the theory of elasticity, where for $d=2, k=1$ the fields $v,w$ are interpreted, respectively, as the out-of-plane and in-plane displacements of the midsurface $\omega$ of a thin elastic plate \cite{lew_book}.
The system (\ref{e:VK}) encodes the agreement, in the leading order, between the perturbed metric $g_\epsilon = \mathrm{Id}_d + 2\epsilon^2 A $ and the induced metric of the augmented embeddings $\bar u_\epsilon(x) = (x + \epsilon^2 w(x) , \epsilon v(x))$, revealing a close connection between the Monge-Amp\`ere system and the problem of isometrically embedding $d$-dimensional Riemannian manifolds into $\R^{d+k}$.

\subsection{The main results}
In this paper, we focus on the case $d=k=2$, and study \emph{flexibility} of solutions to \eqref{e:VK} and weak solutions to (\ref{e:MAdk}). Our main result is the following convex integration theorem, asserting that strict subsolutions to the von K\'arman system can be uniformly approximated by exact solutions with the H\"older $\mathcal{C}^{1,\alpha}$ regularity determined in \eqref{VKrange}: 
% only by the smoothness of the right hand side field $A\in\mathcal{C}^{s,\beta}$:

\begin{theorem}\label{th_final}
Let $\omega\subset\R^2$ be an open, bounded domain. 
Given vector fields $v, w\in\mathcal{C}^1(\bar\omega,\R^2)$,
and a matrix field $A\in\mathcal{C}^{s,\beta}(\bar\omega,\R^{2\times
  2}_\sym)$ with $s\in\{0,1\}$, $\beta\in (0,1]$, $s+\beta<2$, assume that:
\begin{equation}\label{subsol}
\frac{1}{2}(\nabla v)^T\nabla v + \sym\nabla
w < A \quad \; \mbox{ in } \; \bar\omega,
\end{equation}
in the sense of matrix inequalities. Then, for every exponent $\alpha$ satisfying:
\begin{equation}\label{VKrange} 
\alpha<\min\Big\{\frac{s+\beta}{2}, 1 \Big\},
\end{equation}
and for every $\epsilon>0$,  there exists $\tilde v, \tilde w\in\mathcal{C}^{1,\alpha}(\bar
\omega,\R^2)$, such that the following holds:
\begin{align*}
& \|\tilde v - v\|_0\leq \epsilon, \qquad \|\tilde w - w\|_0\leq \epsilon,
\nonumber \vspace{1mm}\\ 
& A - \big(\frac{1}{2}(\nabla \tilde v)^T\nabla \tilde v + \sym\nabla
\tilde w\big) =0 \quad \mbox{ in }\;\bar\omega. \nonumber
\end{align*}
\end{theorem}

\smallskip

\noindent As a byproduct, we obtain the density, in the space of continuous functions, of H\"older 
solutions to the $d=k=2$ version of (\ref{e:MAdk}). Indeed, because of the symmetries in the table-valued operator $\mathfrak{Det}\,\nabla^2$, the Monge-Amp\`ere system reduces to the single equation (\ref{MA2d2k}) below, and we have:

\begin{corollary}\label{th_weakMA} 
For any $f\in L^{\infty} (\omega, \R)$ defined on an open, bounded, simply connected
domain $\omega\subset\mathbb{R}^2$ and for any exponent $\alpha\in (0,1)$, the set of $\mathcal{\mathcal{C}}^{1,\alpha}(\bar\omega, \R^2)$
weak solutions to:
\begin{equation}\label{MA2d2k} 
\det\nabla^2v^1 + \det\nabla^2v^2 = \langle \partial^2_{11}v, \partial^2_{22}v\rangle  - |\partial^2_{12}v|^2 = f \quad \mbox{ in } \; \omega,
\end{equation}
is dense in $\mathcal{\mathcal{C}}^0(\bar\omega, \R^2)$. Namely, every $v\in \mathcal{\mathcal{C}}^0(\bar\omega,\R^2)$ is the
uniform limit of some sequence
$\{v_n\in\mathcal{\mathcal{C}}^{1,\alpha}(\bar\omega,\R^2)\}_{n=1}^\infty$, such that: 
$\langle \partial^2_{11}v_n, \partial^2_{22}v_n\rangle  - |\partial^2_{12}v_n|^2 = f$ on $\omega$, for all $n\geq 1$.
\end{corollary} 

\medskip

\noindent We expect that, using the
corrugation ansatz of \cite{CHI2}, 
one can prove a similar result for isometric embeddings as well. Recall that, according to Poznyak's theorem (see discussion in
\cite[Chapter 2.3]{HH}), any smooth $2$-dimensional Riemannian metric has a
smooth local embedding in $\R^4$. The conjectured parallel density result, albeit only valid for H\"older
continuous solutions, would be stronger in the following sense: rather than
yielding existence of a single solution, it
would yield that an arbitrary subsolution can be
approximated by a $\mathcal{C}^{1,\alpha}$ solution.

\subsection{Comparison with the literature} Our present results, in addition to addressing the well-posedness of the important geometrical system (\ref{e:MAdk}), contributes to the growing body of work on the flexibility of $\mathcal{C}^{1,\alpha}$ isometric immersions and solutions to related PDEs, inspired by the foundational work of Nash and Kuiper \cite{Nash2, Kuiper}. Flexibility results analogous to Theorem \ref{th_final} for the most constrained case of codimension one ($k=1$) have been established in \cite{lewpak_MA} for $d=2$, in \cite{lew_conv} for  arbitrary $d\geq 2, k\geq 1$, in \cite{lew_improved,lew_improved2} for $d=2$ and arbitrary $k\geq 1$, and in \cite{in_lew} for $d=2, k=3$. For results on codimension-one isometric immersions, we refer to \cite{Borisov1965,Borisov2004,CDS,DIS1/5,CHI}. See also \cite{Li} for a flexibility result pertaining to the closely related $2$-Hessian equation. 

\medskip

\noindent In this work, we focus on the case $k=2$ and show that, surprisingly, for the dimension $d=2$ this case can be already considered as a "high codimension". Generally speaking, the higher the codimension, the more flexibility one expects, due to the system being less overdetermined, whereas the additional degrees of freedom manifest themselves as the admissibility of solutions of higher regularity. The case $k=2$ for system \eqref{e:VK} was previously studied in \cite{lew_conv} for arbitrary $d\geq 2$, and for the specific case $d=2$, it was included in the considerations of \cite{lew_improved,lew_improved2}. There, it was shown that Theorem \ref{th_final} holds with the restricted range $\alpha<\min\{\frac{\beta}{2}, \frac37\}$. 

\medskip

\noindent For the isometric immersion problem with $d=k=2$, Poznyak showed in \cite{Poz} that every smooth metric $g$ on $\bar B_R\subset \R^2$ admits a smooth isometric embedding $u:(\bar B_R,g)\to \R^4$. However, his proof does not extend to metrics $g\notin \mathcal{C}^{3,\alpha}$, nor does it establish flexibility in the sense of our main theorem. On the other hand, in \cite{Kallen} K\"all\'en considered $\mathcal{C}^{1,\alpha}$ isometric embeddings for rougher metrics $g\in \mathcal{C}^{\beta}$ with $0<\beta\leq 2$ in high codimension ($k\gg 1$), and proved flexibility in $\mathcal{C}^{1,\alpha}$ for any $\alpha<\min\{\frac{\beta}{2},1\}$. This is what we call "full flexibility", in contrast to classical rigidity theorems for $\mathcal{C}^2$ isometric immersions and $\mathcal{C}^2$ solutions of \eqref{e:MA}.
The codimension $k\gg 1$ in \cite{Kallen} depends on the dimension $d$ of the manifold, and an inspection of the proof shows that for the local version, $k\geq d(d+1)$ is sufficient. 
For $d=2$, an adaptation of K\"all\'en's proof, combined with a suitable coordinate transformation (as in \cite{DIS1/5} for the isometric embedding problem) or an additional perturbation on the domain (as in \cite{CS} for the Monge-Amp\`ere equation), shows that $k\geq 4$ is sufficient. A detailed proof in the Monge-Amp\`ere setting can be found in \cite{lew_improved2}. 
The central contribution of our work is therefore to show that in the two-dimensional case, this kind of "full flexibility" already occurs for codimension $k=2$. 

\smallskip

\subsection{Overview of the main estimates}

Our main technical contribution is given in terms of estimates gathered in the following "stage" construction, whose iteration via the Nash-Kuiper algorithm ultimately provides the proof of Theorem \ref{th_final}:

\begin{theorem}\label{thm_stage} 
Let $\omega\subset\R^2$ be an open, bounded, smooth planar domain. 
Fix two integers $N, K\geq 1$ and an exponent $\gamma\in (0,1)$.
%that is sufficiently small in function of $N$ and $K$. 
Then, there exists $l_0\in (0,1)$ depending only on $\omega$, and there exists
$\sigma_0\geq 1$ depending on $\omega,\gamma, N, K$,
such that the following holds. Given the fields
$v,w\in\mathcal{C}^2(\bar\omega+\bar B_{2l}(0),\R^2)$, 
$A\in\mathcal{C}^{s,\beta}(\bar\omega+\bar B_{2l}(0),\R^{2\times
  2}_\sym)$ with $s\in \{0,1\}$, $\beta\in (0,1]$, defined on the
closed $2l$-neighbourhood of $\omega$, and given the positive constants $l,
\lambda, \mathcal{M}$ with the properties: 
\begin{equation}\label{Assu}
l\leq l_0,\qquad \lambda^{1-\gamma} l\geq\sigma_0, \qquad
\mathcal{M}\geq\max\{\|v\|_2, \|w\|_2, 1\},
\end{equation}
there exist $\tilde v, \tilde w\in\mathcal{C}^2(\bar \omega+\bar B_{l}(0),\R^2)$,
such that, denoting the defects:
\begin{equation}\label{defects}
\mathcal{D}=A -\big(\frac{1}{2}(\nabla v)^T\nabla v + \sym\nabla
w\big), \qquad \tilde{\mathcal{D}} =A -\big(\frac{1}{2}(\nabla \tilde
v)^T\nabla \tilde v + \sym\nabla\tilde w\big), 
\end{equation}
the following bounds are valid:
\begin{align*}
& \begin{array}{l}
\|\tilde v - v\|_1\leq C\lambda^{\gamma/2}\big(\|\mathcal{D}\|_0^{1/2}
+ l\mathcal{M}\big), \vspace{1.5mm}\\ 
\|\tilde w -w\|_1\leq C\lambda^{\gamma}\big(\|\mathcal{D}\|_0^{1/2}
+ l\mathcal{M}\big) \big(1+ \|\mathcal{D}\|_0^{1/2} + l\mathcal{M} +\|\nabla v\|_0\big), 
\end{array}\vspace{3mm} \tag*{(\theequation)$_1$}\refstepcounter{equation} \label{Abound12}\\
&  \begin{array}{l}
\|\nabla^2\tilde v\|_0\leq C \, \displaystyle{\frac{(\lambda
  l)^{K+N}\lambda^{\gamma/2}}{l}}\big(\|\mathcal{D}\|_0^{1/2} + l\mathcal{M}\big), 
\vspace{1.5mm}\\ 
\|\nabla^2\tilde w\|_0\leq C \, \displaystyle{\frac{(\lambda
    l)^{K+N}\lambda^{\gamma}}{l}} 
\big(\|\mathcal{D}\|_0^{1/2} + l\mathcal{M}\big) 
\big(1+\|\mathcal{D}\|_0^{1/2} + l\mathcal{M}+ \|\nabla v\|_0\big),  
\end{array} \vspace{3mm} \tag*{(\theequation)$_2$}\label{Abound22} \\ 
& \begin{array}{l}
\|\tilde{\mathcal{D}}\|_0\leq C\Big(l^{s+\beta} \|A\|_{s,\beta} +
\displaystyle{\frac{\|\mathcal{D}\|_0 +(l\mathcal{M})^2}{(\lambda
    l)^{KN}}}\Big). 
\end{array} \tag*{(\theequation)$_3$} \label{Abound32}
\end{align*}
Above, the norms of the maps $v, w, A, \mathcal{D}$ and $\tilde v,
\tilde w, \tilde{\mathcal{D}}$ in \ref{Abound12} - \ref{Abound32}
are taken on the respective domains of the maps' definiteness.
The constants $C$ depend only on $\omega, \gamma, N, K$.
\end{theorem}

\smallskip

\noindent 
We anticipate that assigning $N$ sufficiently large, the quotient $r=r_{K,N}$ of the
blow-up rate of $\|\nabla^2\tilde v\|_0$
with respect to the rate of decay of $\|\tilde{\mathcal{D}}\|_0$, can
be taken arbitrarily close to $1/K$, whereas
this last quotient approaches $0$ for large $K$:
\begin{equation*}
\lim_{K\to\infty}\lim_{N\to\infty} r_{K,N} =
\lim_{K\to\infty}\lim_{N\to\infty} \frac{K+N}{KN} = 
\lim_{K\to\infty}\frac{1}{K} = 0.
\end{equation*}
Since the H\"older regularity exponent deduced from iterating the ``stage'' as
in Theorem \ref{thm_stage} depends only on $r$ and in fact it equals $\frac{1}{1+2r}$ (see
Theorem \ref{th_NK}), this
implies the claimed range (\ref{VKrange}). We now describe the new ideas allowing to attain this arbitrarily small quotient.

\smallskip

\subsection{Overview of the strategy of proofs} \label{sub_over}

For an extensive description and comparison of the different techniques used in proving flexibility of the Monge-Amp\`ere system, we refer to \cite{in_lew}. In the following, we only present the new contributions of the present work. As customary for $d = 2$ (see \cite{DIS1/5, CS, lew_improved, lew_improved2, in_lew}), the new fields $\tilde v, \tilde w$ in Theorem \ref{thm_stage} are constructed from $v, w$ by first diagonalizing the associated defect $\mathcal{D}$ modulo a symmetric gradient:
\begin{equation*}
\mathcal{D} \doteq A- (\frac{1}{2}(\nabla v)^T\nabla v + \sym\nabla w) = a^2 \mathrm{Id}_2 - \mathrm{sym} \nabla \Psi,
\end{equation*}
see the details of this decomposition in Lemma \ref{lem_diagonal}.
Then, two perturbations are introduced in the form of highly oscillatory corrugations, designed to replace, respectively, each of the two rank-one components of the diagonalized defect: $a^2 e_1 \otimes e_1$ and $a^2 e_2 \otimes e_2$, by the lower order defects $E^1$ and $E^2$. These perturbations are added to $v$ in distinct codimensions, according to the following ansatz, where the frequencies $\lambda \leq \mu$ and the secondary amplitude $b$ are be chosen appropriately, and where $\Gamma$ is a suitable periodic function (see Lemma \ref{lem_step2}):
\begin{equation}\label{ansatz}
\tilde v = v + \frac{a}{\lambda} \Gamma(\lambda x_1) e_1 + \frac{b}{\mu} \Gamma(\mu x_2) e_2.
\end{equation}
 The matching perturbations added to $w$ are as in \cite{lew_conv}.
We note that the ansatz (\ref{ansatz}) leads to an "unbalanced" map $\tilde v$, whose component $\tilde v^1$ oscillates rapidly in the $x_1$ direction and slowly $x_2$, while the second component $\tilde v^2$ behaves conversely. This observation is a point of departure for the new construction in this paper, ultimately leading to Theorem \ref{thm_stage}, as we now explain.

\medskip

\noindent Our first novel contribution is the observation, parallel to that in \cite{CHI2}, that refining the ansatz for the perturbation in $w$ enables a partial cancellation of the errors $E^1$ and $E^2$. Specifically, all but one of the terms in $E^1$ take the form $\gamma(\lambda x_1) H$, where $\gamma$ is a periodic function with zero mean, and $H$ is a symmetric matrix field oscillating at a much lower frequency than $\lambda$. Through an "integration by parts" argument (see Lemma \ref{lem_IBP}, resp. \cite[Proposition 2.4]{CHI2}) we express, up to an error of arbitrarily small size:
\begin{equation*}
\gamma H \approx \mathrm{sym}\, \nabla w_c + G e_2 \otimes e_2, 
\end{equation*}
for a suitable vector field $w_c$. Subtracting $w_c$ from the original ansatz on $\tilde w$ thus cancels $\gamma H$, up to a new error term $G e_2 \otimes e_2$, plus the lower order terms. The new error has the same size as $\gamma H$, however, it can be canceled exactly by choosing the amplitude of the second perturbation to be $b = \sqrt{a^2 - G}$. This effectively allows for the removal of the first rank-one component $a^2 e_1 \otimes e_1$ in the original defect, with almost no extra error.
The same procedure is then carried out for the second perturbation. All but one of the terms in $E^2$ take the form $\tilde{\gamma}(\lambda x_2) \tilde{H}$, where $\tilde{\gamma}$ and $\tilde{H}$ enjoy analogous properties as $\gamma$ and $H$. 
Again, we write:
\begin{equation}
\tilde{\gamma} \tilde{H} \approx \mathrm{sym}\,\nabla \tilde{w}_c + \tilde{G} e_1 \otimes e_1.
\end{equation}
Augmenting the ansatz for $\tilde{w}$ with $\tilde{w}_c$ cancels $\tilde{H}$ up to the error term $\tilde{G} e_1 \otimes e_1$. Here, the second codimension plays a critical role: due to the unbalanced nature of $\tilde{v}$ from equation (\ref{ansatz}), the magnitude of $\tilde{G}$ is significantly smaller than that of $\tilde{H}$. Initially, when perturbations are added only once, this effect is not yet visible, as $v$ satisfies balanced estimates and $\lambda$ can be chosen arbitrarily close to the oscillation frequency of $v$ (which makes the map $v+\frac{a}{\lambda} \Gamma(\lambda x_1) e_1 $ still balanced). However, $\tilde {v}$ now possesses unbalanced estimates, so upon iterating the described procedure starting from $\tilde {v}$, the corresponding $\tilde{G}$ exhibits improved estimates.

\medskip

\noindent Our second novel contribution, %following the approach in \cite{in_lew}, 
is that repeating the above procedure (whose single iteration is presented in Proposition \ref{prop2}), for the total of $K$ times and across $K$ pairs of codimensions, yields the rapid improvement of the estimates in the "stage" Theorem \ref{thm_stage}.
Along the way, we observe that one error term in $E^1$ remains inaccessible to the integration by parts argument, namely a term of the form $\gamma^2 H$, since its oscillating factor does not have zero mean over a period. This term is handled by decomposing it as:
\begin{equation*}
\gamma^2 H = \Big(\gamma^2 - \fint \gamma^2\;\mbox{d}t \Big) H + \fint \gamma^2\;\mbox{d}t\; H.
\end{equation*}
The first term in the right hand side above is managed again via the
"integration by parts" argument, while the second term is absorbed
into the decomposition in Lemma \ref{prop1}, using a version of the K\"all\'en iteration.

\subsection{Organization of the paper and notation} 

In section \ref{sec_step} we present the preparatory lemmas on the mollification and commutator estimates, the basic Kuiper corrugation step construction, and the integration by parts decomposition statements. Section \ref{sec_kallen} contains a simple version of the K\"all\'en's iteration, written for the general scaled defect fields. The iteration counter equals $N$ as in the statement of Theorem \ref{thm_stage}. Section \ref{sec_stage_prep} is devoted to our new "step" construction in which we utilize the two codimensions to reduce the defect via the two aforementioned decomposition techniques. In section \ref{sec_khamsa} we iterate such steps towards a "stage" for a total of $K$ times, going over $K$ pairs of codimensions and proving Theorem \ref{thm_stage}. Finally, section \ref{sec4} recalls the Nash-Kuiper iteration scheme and provides the proof of Theorem \ref{th_final}.

\medskip

\noindent By $\mathbb{R}^{2\times 2}_{\sym}$ we denote the space of symmetric
$2\times 2$ matrices. The space of H\"older continuous vector fields
$\mathcal{C}^{m,\alpha}(\bar\omega,\R^k)$ consists of restrictions of
all $f\in \mathcal{C}^{m,\alpha}(\mathbb{R}^2,\R^k)$ to the closure of
an open, bounded domain $\omega\subset\R^2$. The
$\mathcal{C}^m(\bar\omega,\R^k)$ norm of this restriction is
denoted by $\|f\|_m$, while its H\"older norm in $\mathcal{C}^{m,
  \gamma}(\bar\omega,\R^k)$ is $\|f\|_{m,\gamma}$. 
By $C$ we denote a universal constant that may change from line to
line, but it depends only on the specified parameters.

\section{Four preparatory statements}\label{sec_step}

In this section, we exhibit the four constructions that will be used in the course of our convex
integration algorithm. These are: the approximation estimates, the corrugation definition, and three (two, up to symmetries) decomposition constructions. The first lemma below gathers the convolution estimates and
the commutator estimate from \cite{CDS}:

\begin{lemma}\label{lem_stima}
Let $\phi\in\mathcal{C}_c^\infty(\R^2,\mathbb{R})$ be a standard
mollifier that is nonnegative, radially symmetric, supported on the
unit ball $B(0,1)\subset\R^2$ and such that $\int_{\mathbb{R}^d} \phi \dx = 1$. Denote: 
$$\phi_l (x) = \frac{1}{l^d}\phi(\frac{x}{l})\quad\mbox{ for all
}\; l\in (0,1], \;  x\in\R^2.$$
Then, for every $f,g\in\mathcal{C}^0(\mathbb{R}^2,\R)$, every
$m\geq 0$ and $s\in\{0,1\}$, $\beta\in (0,1]$, there holds:
\begin{align*}
& \|\nabla^{(m)}(f\ast\phi_l)\|_{0} \leq
\frac{C}{l^m}\|f\|_0,\tag*{(\theequation)$_1$}\vspace{1mm} \refstepcounter{equation} \label{stima1}\\
& \|f - f\ast\phi_l\|_0\leq l^{s
+\beta} \|f\|_{s,\beta},\tag*{(\theequation)$_2$} \vspace{1mm} \label{stima2}\\
& \|\nabla^{(m)}\big((fg)\ast\phi_l - (f\ast\phi_l)
(g\ast\phi_l)\big)\|_0\leq {C}{l^{2- m}}\|\nabla f\|_{0}
\|\nabla g\|_{0}, \tag*{(\theequation)$_3$} \label{stima4}
\end{align*}
with constants $C>0$ depending only on $m$. The bound \ref{stima2} is a simplified version of:
$$ \|f - f\ast\phi_l\|_0 \leq \min\big\{l^2\|\nabla^{2}f\|_0, l^{1+\beta}\|\nabla f\|_{0,\beta}, l\|\nabla f\|_0, {l^\beta}\|f\|_{0,\beta}\big\}.$$
\end{lemma}

\medskip

\noindent The next result is specific to dimension $d=2$. We reformulate \cite[Proposition 3.1]{CS}, see also \cite[Lemma 2.3]{lew_improved}, obtaining the defect decomposition with components that enjoy the natural "equidistributed" elliptic estimates:

\begin{lemma}\label{lem_diagonal}
Let $\omega\subset\R^2$ be an open, bounded, Lipschitz set. There exist maps: 
$$\bar\Psi: L^2(\omega,\R^{2\times
  2}_\sym)\to W^{1,2}(\omega,\R^2), \qquad \bar a: L^2(\omega,\R^{2\times  2}_\sym)\to L^{2}(\omega,\R), $$ 
which are linear, continuous, and such that:
\begin{itemize}
\item[(i)] for all $H\in L^2(\omega,\R^{2\times 2}_\sym)$ there holds:
  $H+ \sym\nabla \big(\bar\Psi(H)\big) = \bar a(H)\Id_2$,\vspace{1mm}
\item[(ii)] $\bar\Psi(\Id_2) \equiv 0$ and $\bar a(\Id_2) \equiv 1$ in
  $\omega$, \vspace{1mm}
\item[(iii)] for all $m\geq 0$ and $\gamma\in (0,1)$, if $\omega$ is
  $\mathcal{C}^{m+2,\gamma}$ regular then the maps $\bar\Psi$ and $\bar a$ are continuous from
  $\mathcal{C}^{m,\gamma}(\bar\omega,\R^{2\times 2}_\sym)$ to
  $\mathcal{C}^{m+1,\gamma}(\bar\omega, \R^2)$ and
  to $\mathcal{C}^{m,\gamma}(\bar\omega, \R)$, respectively, so that:
\begin{equation}\label{diag_bounds}
\|\bar\Psi (H)\|_{m+1,\gamma}\leq C \|H\|_{m,\gamma} \mbox{ and }
~ \|\bar a (H)\|_{m,\gamma}\leq C \|H\|_{m,\gamma} \quad \mbox{ for all
}\; H\in L^2(\omega,\R^{2\times 2}_\sym).
\end{equation}
\end{itemize} 
The constants $C$ above depend on $\omega$, $m, \gamma$ but not on
$H$. Also, there exists $l_0>0$ depending only on $\omega$, such that
(\ref{diag_bounds}) are uniform on 
the closed $l$-neighbourhoods $\{\bar\omega+ \bar B_l(0)\}_{l\in (0,l_0)}$ of $\omega$.
\end{lemma}

\medskip

\noindent As the next ingredient, we recall the 
``step'' construction from \cite[Lemma 2.1]{lew_conv}, 
in which a single codimension is used to cancel one rank-one defect of the form
$a(x)^2e_i\otimes e_i$:

\begin{lemma}\label{lem_step2}
Let $v,w\in \mathcal{C}^1(\R^2, \R^{2})$, $\lambda>0$ and $a\in
\mathcal{C}^2(\R^2,\R)$ be given. Denote: 
$$\Gamma(t) = 2\sin t,\quad \bar\Gamma(t) = \frac{1}{2}\cos (2t),\quad
\dbar\Gamma(t) = -\frac{1}{2}\sin (2t),\quad \tbar\Gamma(t)=1- \frac{1}{2}\cos(2t),$$
and for a fixed  $i,k=1, 2$ define:
\begin{equation}\label{defi_per2}
\tilde v = v + \frac{a(x)}{\lambda} \Gamma(\lambda x_i) e_k,\quad 
\tilde w = w -\frac{a(x)}{\lambda} \Gamma(\lambda x_i)\nabla v^k
+ \frac{a(x)}{\lambda^2} \bar\Gamma(\lambda x_i)\nabla a(x)
+ \frac{a(x)^2}{\lambda} \dbar\Gamma(\lambda x_i)e_i.
\end{equation}
Then, the following identity is valid on $\R^2$:
\begin{equation}\label{step_err2}
\begin{split}
& \big(\frac{1}{2}(\nabla \tilde v)^T \nabla \tilde v + \sym\nabla \tilde w\big) - 
\big(\frac{1}{2}(\nabla v)^T \nabla v + \sym\nabla w\big) - a(x)^2e_i\otimes e_i
\\ & = -\frac{a}{\lambda} \Gamma(\lambda x_i)\nabla^2 v^k
+ \frac{a}{\lambda^2} \bar\Gamma(\lambda x_i) \nabla^2 a
+ \frac{1}{\lambda^2}\tbar\Gamma(\lambda x_i)\nabla a\otimes\nabla a.
\end{split}
\end{equation}
\end{lemma}

\medskip

\noindent Finally, the following result is an explicit version of the "integration by parts" argument of \cite{CHI2}, which provides a new ingredient of the proofs with respect to our prior observations in \cite{lew_conv, lew_improved, lew_improved2, in_lew}. It yields a new defect's decomposition, complementary to that in Lemma \ref{lem_diagonal}, through removing a further symmetric gradient from the given oscillatory component of a defect at hand and reducing it to a defect of higher order in the frequency, plus another term that agrees in the frequency yet has a lower dimensionality rank. Namely, we have:

\begin{lemma}\label{lem_IBP}
Given  $H\in\mathcal{C}^{k+1}(\R^2,\R^{2\times 2}_\sym)$, $\lambda>0$,
and $\Gamma_0\in\mathcal{C}(\R,\R)$, we have the decomposition:
\begin{equation}\label{ibp1}
\begin{split}
\frac{\Gamma_0(\lambda x_1)}{\lambda}H = & \; (-1)^{k+1}\frac{\Gamma_{k+1}(\lambda
  x_1)}{\lambda^{k+2}} \sym\nabla L_k  \\ & + \sym\nabla
\Big(\sum_{i=0}^k(-1)^i\frac{\Gamma_{i+1}(\lambda x_1)}{\lambda^{i+2}}L_i\Big) 
+ \Big( \sum_{i=0}^{k}(-1)^i\frac{\Gamma_i(\lambda
  x_1)}{\lambda^{i+1}} P_i\Big) e_2\otimes e_2
\end{split}
\end{equation}
where the functions $\Gamma_i\in \mathcal{C}^{i}(\R,\R)$ satisfy the recursive definition:
\begin{equation}\label{recu_ibp}
\Gamma_{i+1}' = \Gamma_{i}  \quad\mbox{ for all } \; i =0\ldots k,
\end{equation}
while $L_i\in\mathcal{C}^{k+1-i}(\R^2,\R^2)$ and
$P_i\in\mathcal{C}^{k+1-i}(\R^2,\R)$ are given in:
\begin{equation}\label{coef_ibp1}
\begin{split}
& \;\; \, L_0 = (H_{11}, 2H_{12}), \qquad P_0 = H_{22},\\
& \left.\begin{array}{l} L_i= (\partial_1^{(i)}H_{11}, 2\partial_1^{(i)}H_{12} +
i\partial_1^{(i-1)}\partial_2H_{11}), \vspace{2mm} 
\\  P_i = 2\partial_1^{(i-1)}\partial_2H_{12} +
(i-1)\partial_1^{(i-2)}\partial_2^{(2)}H_{11} \end{array}\right\}
\quad \mbox{ for all }\; i=1\ldots k.
\end{split}
\end{equation}
In particular, there holds:
\begin{equation*}
\begin{split}
\sym\nabla L_k = & \; \partial_1^{(k+1)}H_{11} e_1\otimes e_1 
+ \big(2\partial_1^{(k+1)}H_{12}+(k+1)\partial_1^{(k)}\partial_2H_{11}
\big) \sym (e_1\otimes e_2) \\
& + \big(2\partial_1^{(k)}\partial_2H_{12} +
k\partial_1^{(k-1)}\partial_2^{(2)}H_{11}  \big) e_2\otimes e_2.
\end{split}
\end{equation*}
\end{lemma}
\begin{proof}
{\bf 1.} We start by checking that for all $i=0\ldots k-1$ there holds:
\begin{equation}\label{naua}
\sym\nabla L_i= \sym  (L_{i+1}\otimes e_1) + P_{i+1}\, e_2\otimes e_2.
\end{equation}
At $i=0$ we have $L_{i+1}=L_1=(\partial_1H_{11},
2\partial_1H_{12}+\partial_2H_{11})$ and $P_{i+1}=P_1=2\partial_2H_{12}$, hence:
\begin{equation*}
\begin{split}
\sym\nabla L_0 & =  \partial_1 H_{11} \,e_1\otimes e_1 +
\big(2\partial_1H_{12} + \partial_2 H_{11}\big)\,\sym (e_1\otimes e_2)
+ 2\partial_2H_{12} \,e_2\otimes e_2\\
& =\sym (L_1\otimes e_1) + P_{1}\, e_2\otimes e_2.
\end{split}
\end{equation*}
On the other hand, for $i\geq 1$ it follows that:
\begin{equation*}
\begin{split}
\sym\nabla L_i & =  \partial_1^{(i+1)} H_{11} \,e_1\otimes e_1 +
\big(2\partial_1^{(i+1)} H_{12} + (i+1)\partial_1^{(i)}\partial_2H_{11}
\big)\,\sym (e_1\otimes e_2) \\ & \quad 
+ \big(2\partial_1^{(i)}\partial_2H_{12} + i\partial_{1}^{(i-1)}\partial_2^{(2)}H_{11}\big)\,e_2\otimes e_2\\
& =\sym (L_{i+1}\otimes e_1) + P_{i+1}\, e_2\otimes e_2.
\end{split}
\end{equation*}
This validates (\ref{naua}). Observe also that (\ref{recu_ibp}) yields:
\begin{equation}\label{daa}
\frac{\Gamma_i(\lambda x_1)}{\lambda^{i+1}}\sym(L_i\otimes e_1) = 
-  \frac{\Gamma_{i+1}(\lambda  x_1)}{\lambda^{i+2}} \sym\nabla L_i 
 + \sym\nabla \Big(\frac{\Gamma_{i+1}(\lambda x_1)}{\lambda^{i+2}}L_i\Big). 
\end{equation}

\smallskip

{\bf 2.} The proof of (\ref{ibp1}) is carried out by induction on $k$. 
At $k=0$, we use (\ref{daa}) together with 
the obvious decomposition $H=\sym(L_0\otimes e_1) + P_0\, e_2\otimes e_2$, to get:
\begin{equation*}
\begin{split}
\frac{\Gamma_0(\lambda x_1)}{\lambda}H 
& = \frac{\Gamma_0(\lambda x_1)}{\lambda} \sym(L_0\otimes e_1) + 
\frac{\Gamma_0(\lambda x_1)}{\lambda} P_0 \, e_2\otimes e_2
\\ & = - \frac{\Gamma_{1}(\lambda x_1)}{\lambda^{2}} \sym\nabla L_0  
+ \sym\nabla \Big(\frac{\Gamma_{1}(\lambda x_1)}{\lambda^{2}}L_0\Big) 
+ \frac{\Gamma_0(\lambda x_1)}{\lambda^{1}} P_0 \, e_2\otimes e_2
\end{split}
\end{equation*}
Assume that (\ref{ibp1}) holds at some $k\geq 0$. The first term in
its right hand side can be rewritten in virtue of (\ref{naua}) and
(\ref{daa}) as: 
\begin{equation*} 
\begin{split}
 & (-1)^{k+1} \frac{\Gamma_{k+1}(\lambda  x_1)}{\lambda^{k+2}} \sym\nabla L_k 
\\ & \qquad =   (-1)^{k+1}\frac{\Gamma_{k+1}(\lambda  x_1)}{\lambda^{k+2}} \sym (L_{k+1}\otimes e_1) +
(-1)^{k+1}\frac{\Gamma_{k+1}(\lambda  x_1)}{\lambda^{k+2}}  P_{k+1}\, e_2\otimes e_2
\\ & \qquad = (-1)^{k+2}\frac{\Gamma_{k+2}(\lambda  x_1)}{\lambda^{k+3}} \sym\nabla L_{k+1}
 +  (-1)^{k+1} \sym\nabla \Big(\frac{\Gamma_{k+2}(\lambda x_1)}{\lambda^{k+3}}L_{k+1}\Big)
\\ & \qquad \quad + (-1)^{k+1}\frac{\Gamma_{k+1}(\lambda
  x_1)}{\lambda^{k+2}}  P_{k+1}\, e_2\otimes e_2. 
\end{split}
\end{equation*}
Summing with the other two terms of (\ref{ibp1}), the identity follows
at $k+1$ as claimed.
\end{proof}

\medskip

\noindent A symmetric decomposition as in Lemma \ref{lem_IBP} holds with respect to $x_2$ rather than $x_1$:

\begin{corollary}\label{cor_IBP}
Let $H$, $\lambda$, $\Gamma_0$ be as in Lemma \ref{lem_IBP} and
$\{\Gamma_i\in \mathcal{C}^{i}(\R,\R)\}_{i=1}^{k+1}$ as in (\ref{recu_ibp}).
Then:
\begin{equation}\label{ibp2}
\begin{split}
\frac{\Gamma_0(\lambda x_2)}{\lambda}H = & \; (-1)^{k+1}\frac{\Gamma_{k+1}(\lambda
  x_2)}{\lambda^{k+2}} \sym\nabla \tilde L_k  \\ & + \sym\nabla
\Big(\sum_{i=0}^k(-1)^i\frac{\Gamma_{i+1}(\lambda
  x_2)}{\lambda^{i+2}}\tilde L_i\Big) 
+ \Big( \sum_{i=0}^{k}(-1)^i\frac{\Gamma_i(\lambda
  x_2)}{\lambda^{i+1}} \tilde P_i\Big) e_1\otimes e_1,
\end{split}
\end{equation}
with $\tilde L_i\in\mathcal{C}^{k+1-i}(\R^2,\R^2)$,
$\tilde P_i\in\mathcal{C}^{k+1-i}(\R^2,\R)$ given in:
\begin{equation}\label{coef_ibp2}
\begin{split}
& \;\; \; \tilde L_0 = (2H_{12}, H_{22}), \qquad \tilde P_0 = H_{11},\\
& \left.\begin{array}{l} \tilde L_i= (2\partial_2^{(i)}H_{12} +
i\partial_1\partial_2^{(i-1)}H_{22}, \partial_2^{(i)}H_{22}), \vspace{2mm} 
\\  \tilde P_i = 2\partial_1\partial_2^{(i-1)}H_{12} +
(i-1)\partial_1^{(2)}\partial_2^{(i-2)}H_{22} \end{array}\right\}
\quad \mbox{ for all }\; i=1\ldots k.
\end{split}
\end{equation}
In particular, there holds:
\begin{equation*}
\begin{split}
\sym\nabla \tilde L_k = & \; \big(2\partial_1\partial_2^{(k)} H_{12} +
k\partial_1^{(2)}\partial_2^{(k-1)}H_{22} \big) e_1\otimes e_1 \\
& + \big(2\partial_2^{(k+1)}H_{12}+(k+1)\partial_1\partial_2^{(k)}H_{22}
\big) \sym (e_1\otimes e_2) + \partial_2^{(k+1)}H_{22} e_2\otimes e_2.
\end{split}
\end{equation*}
\end{corollary}

\medskip

\noindent We remark that if $\Gamma_0$ has the general form $\alpha\sin(\beta t)$ or
$\alpha \cos(\beta t)$, as is the case for the zero-mean periodic profiles $\Gamma$, $\bar\Gamma$
or $\tbar\Gamma -1$ in Lemma \ref{lem_step2}, then the primitives
$\Gamma_i$ are of the same form $\frac{\alpha}{\beta^i}\sin(\beta t)$
or $\frac{\alpha}{\beta^i}\cos(\beta t)$. All their
derivatives are then bounded, which allows the uniformity of estimates
in Proposition \ref{prop2}, and is the reason why we work with $\tbar{\Gamma} - 1$ 
instead of $\tbar \Gamma$.

\section{The K\"all\'en iteration}\label{sec_kallen}

In this section, we carry out a simple version of K\"all\'en's iteration, with the purpose of canceling the non-oscillatory  portion of the last defect term in the right hand side of (\ref{step_err2}), corresponding to $\frac{1}{\lambda^2}\nabla a\otimes\nabla a$. The remaining portion $\frac{1}{\lambda^2}(\tbar\Gamma(\lambda x_1) -1)\nabla a\otimes\nabla a$ will be canceled (in the leading order terms) via an application of Lemma \ref{lem_IBP}. The matrix field $H$ in the statement below should be thought of as the scaled defect $\mathcal{D}$ in the proof of Theorem \ref{thm_stage}.

\begin{proposition}\label{prop1}
Let $\omega\subset\R^2$ be open, bounded, smooth and let $N\geq 1$ and
$\gamma\in (0,1)$. Then, there exists $l_0\in (0,1)$
depending only on $\omega$ and $\sigma_0\geq 2$ depending on $\omega,
\gamma, N$ such that the following holds. Given the constants
$l>0$, $\mu, \bar\lambda>1$ and an integer $M$, satisfying:
\begin{equation}\label{ass_de}
l\leq l_0,\qquad \bar\lambda^{1-\gamma}\geq \mu\sigma_0,\qquad M\geq 0,
\end{equation}
and given the field $H\in\mathcal{C}^\infty(\bar\omega+\bar
B_{l}(0),\R^{2\times 2}_\sym)$ with the property:
\begin{equation}\label{ass_H0}
\|\nabla^{(m)}H\|_0\leq \mu^m \qquad \mbox{for all } \; m=0\ldots M+2N,
\end{equation}
there exist $a\in \mathcal{C}^\infty(\bar\omega + \bar B_l(0),\R)$
and $\Psi\in \mathcal{C}^\infty(\bar\omega + \bar B_l(0),\R^2)$
such that, denoting:
$$\displaystyle{\mathcal{F} = a^2\Id_2 + \sym\nabla\Psi - H +
\frac{1}{\bar\lambda^2}\nabla a\otimes\nabla a,}$$
there hold the estimates:
\begin{equation}\label{Ebounds}
\begin{split}
&  \; \; \; \frac{\bar C}{2}\mu^\gamma\leq a^2\leq \frac{3\bar C}{2}\mu^\gamma
\hspace{5.5mm}\mbox{ and }\quad 
\frac{\bar C^{1/2}}{2}\mu^{\gamma/2}\leq a\leq \frac{3\bar C^{1/2}}{2}\mu^{\gamma/2}, \\
& \; \; \; \|\Psi\|_1\leq C\mu^\gamma \hspace{17.8mm} \mbox{and }\quad
  \|\nabla^2\Psi\|_0\leq C\mu^\gamma\mu,  \vspace{3mm} \\
& \left. \begin{array}{l} \|\nabla^{(m)} a^2\|_0\leq C\mu^{\gamma}\mu^m\quad 
\mbox{and } \quad \|\nabla^{(m)} a\|_0\leq C\mu^{\gamma/2}\mu^m \vspace{2mm}\\ 
\displaystyle{\|\nabla^{(m)}\mathcal{F}\|_0 \leq 
C\frac{\mu^m}{(\bar\lambda/\mu)^{N}}}
%C\mu^{\gamma N}\frac{\mu^m}{(\bar\lambda/\mu)^{2N}}}
\end{array}\right\} \quad\mbox{for all } \; m=0\ldots M. 
\end{split}
\end{equation}
The constant $\bar C$ depends only on $\omega,
\gamma$, while $C$ depends on $\omega,\gamma, N, M$.
\end{proposition}

\begin{proof}
{\bf 1. (Iteration set-up)}  With the help of Lemma
\ref{lem_diagonal}, we will inductively define 
scalar fields $\{a_i\in \mathcal{C}^\infty(\bar\omega + \bar B_l(0),\R)\}_{i=1}^N$ 
and vector fields $\{\Psi_i\in \mathcal{C}^\infty(\bar\omega + \bar 
B_l(0),\R^2)\big\}_{i=1}^N$ such that:
\begin{equation}\label{def_aP}
\begin{split}
& a_i^2\Id_2 + \sym\nabla \Psi_i = H - \mathcal{E}_{i-1}\quad
\; \, \mbox{ for all } \; i=1\ldots N,\\
&\mbox{where } \quad \mathcal{E}_i = \frac{1}{\bar\lambda^2}\nabla a_{i}\otimes\nabla a_{i}
\qquad \mbox{ for all } \; i=0 \ldots N,
\end{split}
\end{equation}
where we have set $ a_0=0$. Declaring:
\begin{equation*}
\begin{split}
& a=a_N,\quad \Psi=\Psi_N, \qquad 
\mbox{so that }\; \mathcal{F} = \mathcal{E}_N - \mathcal{E}_{N-1}, 
\end{split}
\end{equation*}
the bounds (\ref{Ebounds}) will be implied by the following estimate,
that we prove below for $i=1\ldots N$:
\begin{align*}
&  \frac{\bar C}{2}\mu^\gamma\leq a_i^2\leq \frac{3\bar
  C}{2}\mu^\gamma \hspace{3.5mm} \quad\mbox{ and }\quad
\frac{\bar C^{1/2}}{2}\mu^{\gamma/2}\leq a_i\leq \frac{3\bar C^{1/2}}{2}\mu^{\gamma/2}, 
\tag*{(\theequation)$_1$}\refstepcounter{equation} \label{Ebound12}\\
& \|\nabla^{(m)} a_i^2\|_0\leq C\mu^{\gamma}\mu^m \quad \mbox{ and }
\quad \|\nabla^{(m)} a_i\|_0\leq C\mu^{\gamma/2}\mu^m \qquad
\, \mbox{ for all } \; m=1\ldots M,
\vspace{5mm} \tag*{(\theequation)$_2$}\label{Ebound22} \\ 
& \|\Psi_i\|_1\leq C\mu^\gamma \quad \mbox{and}\quad
  \|\nabla^2\Psi_i\|_0\leq C\mu^\gamma\mu,  \vspace{3mm} 
\tag*{(\theequation)$_3$} \label{Ebound42}\\
& \|\nabla^{(m)}\big( \mathcal{E}_i - \mathcal{E}_{i-1}\big)\|_0 
\leq C \mu^{\gamma i} \frac{\mu^{m}}{(\bar\lambda/\mu)^{2i}}
\quad \qquad \; \;\, \mbox{ for all }\; m=0\ldots M.
\tag*{(\theequation)$_4$} \label{Ebound52}
\end{align*}
The constant $\bar C$ in \ref{Ebound12} will be shown to depend only on $\omega,
\gamma$. Other constants $C$ depend: in \ref{Ebound22}
on $\omega,\gamma, n, M$ and only $M+2(i-1)+1$ derivatives of $H$;
in \ref{Ebound42} on $\omega,\gamma, i$ and only $2+2(i-1)$derivatives
of $H$; in \ref{Ebound52} on $\omega, \gamma, N, M$ necessitating
the full condition (\ref{ass_H0}). The last bound in (\ref{Ebounds})
follows from \ref{Ebound52} and the second condition in (\ref{ass_de}), because:
$$\displaystyle{\|\nabla^{(m)}\mathcal{F}\|_0 =
  \|\nabla^{(m)}\big(\mathcal{E}_N-\mathcal{E}_{N-1}\big)\|_0
\leq C\mu^{\gamma N}\frac{\mu^m}{(\bar\lambda/\mu)^{2N}}}
\leq C \Big(\frac{\bar\lambda^\gamma}{\bar\lambda/\mu}\Big)^N
\frac{\mu^m}{(\bar\lambda/\mu)^N} \leq C \frac{\mu^m}{(\bar\lambda/\mu)^N}.$$ 

\medskip

We start by observing that the second bound in \ref{Ebound12} is 
implied by the first one, because:
$$\big|\frac{a_i(x)}{\bar C^{1/2}\mu^{\gamma/2}}-1\big|\leq
\big|\frac{a_i(x)^2}{\bar C\mu^{\gamma}}-1\big|\leq \frac{1}{2}.$$
Likewise, the second bound in \ref{Ebound22} follows from the first one in view of 
\ref{Ebound12}, in virtue of the Fa\'a di Bruno formula:
\begin{equation*}
\begin{split}
\|\nabla^{(m)}a_i\|_0 & \leq C\Big\|\sum_{p_1+2p_2+\ldots
  mp_m=m} a_i^{2(1/2-p_1-\ldots -p_m)}\prod_{z=1}^m\big|\nabla^{(z)}a_i^2\big|^{p_z}\Big\|_0
\\ & \leq C\|a_i\|_0 \sum_{p_1+2p_2+\ldots
  mp_m=m}\prod_{z=1}^m\Big(\frac{\|\nabla^{(z)}a_i^2\|_0}{\bar
  C \mu^\gamma}\Big)^{p_z}\leq C\mu^{\gamma/2}\mu^m,
\end{split}
\end{equation*}
necessitating the same number of derivatives bounds in condition (\ref{ass_H0}). 
Applying Fa\'a di Bruno's formula to the inverse rather than the square root, \ref{Ebound22} yields:
\begin{equation}\label{asm2}
\|\nabla^{(m)}\Big(\frac{1}{a_{i+1}+a_{i}}\Big)\|_0  \leq
 \frac{C}{\mu^{\gamma/2}}\sum_{p_1+2p_2+\ldots
  mp_m=m}\prod_{z=1}^m\Big(\frac{\|\nabla^{(z)}(a_{i+1}+a_{i})\|_0}{\bar
  C^{1/2}\mu^{\gamma/2}}\Big)^{p_z}\leq \frac{C}{\mu^{\gamma/2}}\mu^m, 
\end{equation}
with the the same dependence of constants and using the same number of
derivatives of $H$ as in the corresponding bounds on
$a_i$ and $a_{i+1}$.

\smallskip

{\bf 2. (Induction base $i=1$ and definition of $\bar C$)}
Let the linear maps ${\bar a, \bar\Psi}$ be as in 
Lemma \ref{lem_diagonal}, applied with the specified $\gamma$.
From the bound on $\|H\|_1$ in (\ref{ass_H0}), we obtain:
$$\|\bar a(H)\|_0\leq C\|H\|_{0,\gamma}\leq C\big(\|H\|_0 +
\|H\|_0^{1-\gamma} \|\nabla H\|_0^\gamma \big)\leq C\mu^\gamma$$
where $C$ depends on $\omega$, $\gamma$. We declare $\bar C$ to
be four times the final constant above, leading to:
\begin{equation}\label{Ct_st}
\|\bar a(H)\|_0\leq \frac{\bar C}{4}\mu^\gamma.
\end{equation}
This results in the validity of the first bound in \ref{Ebound12}, where we set:
\begin{equation}\label{def_aP_global1}
a_1^2 = \bar C\mu^\gamma + \bar a(H), \qquad \Psi_1=\bar\Psi(H) - \bar C\mu^\gamma id_2,
\end{equation}
while the identity (\ref{def_aP}) holds because $\mathcal{E}_0=0$ and:
\begin{equation*}
\begin{split}
H & = \bar a(H)\Id_2 + \sym \nabla \big(\bar\Psi(H)\big) 
\\ & = (a_1^2\Id_2 - \bar C\mu^\gamma \Id_2 ) + (\sym\nabla \Psi_1 + \bar
C\mu^\gamma \Id_2) = a_1^2\Id_2 + \sym\nabla \Psi_1.
\end{split}
\end{equation*}
Further, using the bound on $\|H\|_{M+1}$ in (\ref{ass_H0}), we obtain for all $m=1\ldots M$:
\begin{equation}\label{a1_st}
\begin{split}
\|\nabla^{(m)}a_1^2\|_0 & \leq C \|H\|_{m,\gamma} \leq C \big(
\|H\|_m +  \|\nabla^{(m)}H\|_0^{1-\gamma}
\|\nabla^{(m+1)}H\|_0^\gamma\big) \\ & \leq C\big(\mu^m +
\mu^{m(1-\gamma)}\mu^{(m+1)\gamma}\big) \leq C \mu^\gamma\mu^m,
\end{split}
\end{equation}
where $C$ depends on $\omega, \gamma, M$. The above implies
the first bound in \ref{Ebound22}. Towards \ref{Ebound52}, we apply
(\ref{a1_st}) and use (\ref{ass_H0}) up to $\|H\|_{M+2}$ in:
\begin{equation*}
\big\|\nabla^{(m)}\mathcal{E}_1\big\|_0 \leq  
C\sum_{p+q=m}\bar\lambda^{-2}\|\nabla^{(p+1)} a_1\|_0\|\nabla^{(q+1)}a_1\|_0 
\leq C\sum_{p+q=m}\bar\lambda^{-2}\mu^\gamma\mu^{p+q+2}
\leq C \mu^\gamma \frac{\mu^m}{(\bar\lambda/\mu)^2},
\end{equation*}
valid for $m\leq M$, with $C$ depending on $\omega, \gamma, M$.
The above is precisely \ref{Ebound52} since $\mathcal{E}_0=0$.

\smallskip

{\bf 3. (Induction step: bounds \ref{Ebound12}, \ref{Ebound22})} Assume that \ref{Ebound12},
\ref{Ebound22}, \ref{Ebound52}  hold up to some $1\leq i\leq N-1$,
necessitating (\ref{ass_H0}) to estimate derivatives of $H$ only up to
${M+2i}$. We will prove the validity of the same bounds
at $i+1$, necessitating (\ref{ass_H0}) to estimate
$\|H\|_{M+2i+1}$. We start by noting that, as a consequence of
\ref{Ebound52}, for all $j=1\ldots i$ and $m=0\ldots M$ we have:
\begin{equation}\label{imtih1}
\|\nabla^{(m)}\big(\mathcal{E}_j -\mathcal{E}_{j-1}\big)\|_{0,\gamma}
\leq C\mu^{\gamma (j+1)}\frac{\mu^m}{(\bar\lambda/\mu)^{2j}}
= C\mu^\gamma \mu^m\Big(\frac{\mu^\gamma}{(\bar\lambda/\mu)^{2}}\Big)^j,
\end{equation}
necessitating the bounds on $\|H\|_{M+2i+1}$ and with $C$ depending on
$\omega,\gamma, i, M$. This yields:
\begin{equation}\label{imtih2}
\|\nabla^{(m)}\mathcal{E}_i\|_{0,\gamma}\leq C\mu^\gamma \mu^m
\sum_{j=1}^i \Big(\frac{\mu^\gamma}{(\bar\lambda/\mu)^{2}}\Big)^j \leq C \mu^\gamma
\mu^m\frac{\mu^\gamma/(\bar\lambda/\mu)^2}
{1-\mu^\gamma/(\bar\lambda/\mu)^2} \leq  C \mu^{2\gamma}
\frac{\mu^m}{(\bar\lambda/\mu)^2}  
\end{equation}
as $\mu^\gamma/(\bar\lambda/\mu)^2\leq
\bar\lambda^\gamma/(\bar\lambda/\mu) \leq 1/\sigma_0\leq 1/2$ from the second
assumption in (\ref{ass_de}). In particular:
$$\|\bar a(\mathcal{E}_i)\|_0\leq C\|\mathcal{E}_i\|_{0,\gamma}\leq C
\frac{\mu^{2\gamma}}{(\bar \lambda/\mu)^2} \leq \frac{\bar C}{4}\mu^\gamma $$
provided that $\sigma_0$ is large enough, in function of $\omega, \gamma, i$.
Recalling (\ref{Ct_st}), the above yields the well definiteness of
$a_{i+1}^2$ together with the first bound in \ref{Ebound12}, upon defining:
\begin{equation}\label{def_aP_global}
a_{i+1}^2 = \bar C\mu^\gamma + \bar a(H -\mathcal{E}_i),\qquad 
\Psi_{i+1} = \bar\Psi(H - \mathcal{E}_i) - \bar C\mu^\gamma id_2.
\end{equation}
The identity (\ref{def_aP}) clearly holds from Lemma \ref{lem_diagonal} (i).
Towards proving \ref{Ebound22}, we apply (\ref{imtih2}) and the bound on
$\|H\|_{m,\gamma}$ in (\ref{a1_st}), to get:
$$\| \nabla^{(m)}a_{i+1}^2\|_0\leq C \|H -\mathcal{E}_i\|_{m,\gamma} 
\leq C\mu^\gamma\mu^m\Big(1+ 
\frac{\mu^\gamma}{(\bar\lambda/\mu)^2}\Big) \leq C\mu^\gamma\mu^m$$
for all $m=1\ldots M$. Above the constant $C$ depends on $\omega,\gamma, i,
M$ and condition (\ref{ass_H0}) has been used 
only up to ${M+2i}+1$ derivatives of $H$.

\smallskip

{\bf 4. (Induction step: the bound \ref{Ebound52})} We
continue the inductive step argument and show \ref{Ebound52} at $i+1$,
necessitating (\ref{ass_H0}) to estimate 
derivatives of $H$ up to ${M+2(i+1)}$. From the definitions (\ref{def_aP_global1}) and
(\ref{def_aP_global}), it follows that: 
$$a_{i+1}^2 - a_i^2 = -\bar a(\mathcal{E}_i - \mathcal{E}_{i-1}).$$
Consequently, recalling (\ref{imtih1}) we get:
\begin{equation*}
\begin{split}
& \|\nabla^{(m)} (a_{i+1}^2 - a_i^2)\|_0\leq C\|\mathcal{E}_i - \mathcal{E}_{i-1}\|_{m,\gamma}
\leq C\mu^\gamma \mu^m \big(\frac{\mu^{\gamma i}}{(\bar\lambda/\mu)^{2i}},
\end{split}
\end{equation*}
which together with (\ref{asm2}) implies:
\begin{equation}\label{imtih3}
\begin{split}
\|\nabla^{(m)} (a_{i+1} - a_i)\|_0 & \leq C \sum_{p+q=m} \|\nabla^{(p)}(a_{i+1}^2 - a_i^2)\|_0
\big\|\nabla^{(q)}\big(\frac{1}{a_{i+1}+a_i}\big)\big\|_0 \\ & \leq 
C\mu^{\gamma/2}\mu^m\frac{\mu^{\gamma i}}{(\bar\lambda/\mu)^{2i}}
\end{split}
\end{equation}
for all $m=0\ldots M$, with $C$ depending on
$\omega, \gamma, i, M$ and where we used bounds in (\ref{ass_H0}) on
the derivatives of $H$ up to $M+2i+1$. Towards proving \ref{Ebound52}, 
we use the identity: 
$$\nabla a_{i+1}\otimes \nabla a_{i+1} - \nabla a_{i}\otimes \nabla a_{i} = \nabla
(a_{i+1}-a_i)\otimes \nabla (a_{i+1}-a_i) + 2\,\sym \big(  
\nabla (a_{i+1}-a_i)\otimes \nabla a_{i}\big)$$ 
and estimate, using (\ref{imtih3}), the already established bounds
\ref{Ebound22}, \ref{Ebound22}:
\begin{equation*}
\begin{split}
& \big\|\nabla^{(m)}\big(\mathcal{E}_{i+1}-\mathcal{E}_i\big) \big\|_0 
\\  & \leq C \sum_{p+q=m} \bar \lambda^{-2}\|\nabla^{(p+1)}(a_{i+1}-a_i)\|_0 
\big(\|\nabla^{(q+1)} a_{i+1}\|_0+ \|\nabla^{(q+1)} a_{i}\|_0 \big)\\
& \leq C\frac{\mu^{m+2}}{\bar\lambda^2}\frac{\mu^{\gamma/2}\mu^{\gamma
  i }}{(\bar\lambda/\mu)^{2i}} \mu^{\gamma/2} 
\leq C\mu^{\gamma(i+1)} \frac{\mu^m}{(\bar\lambda/\mu)^{2(i+1)}},
\end{split}
\end{equation*}
which is exactly \ref{Ebound52} at $i+1$, with the right dependence
of constants and order of used derivatives of $H$ (up to $M+2i+2$), as claimed.

\smallskip

{\bf 5. (The bound \ref{Ebound42})} From the definitions
(\ref{def_aP_global1}), (\ref{def_aP_global}) we obtain, for all
$i=1\ldots N$ in virtue of Lemma \ref{lem_diagonal} (iii) and (\ref{imtih2}):
\begin{equation*}
\begin{split}
& \|\Psi_i\|_1\leq C\big(\mu^\gamma + \|H\|_{0,\gamma} + 
\|\mathcal{E}_{i-1}\|_{0,\gamma}\big) \leq C\mu^\gamma,
\\ & \|\nabla^2\Psi_i\|_0\leq C\big(\|H\|_{1,\gamma} + 
\|\mathcal{E}_{i-1}\|_{1,\gamma}\big) \leq C\mu^\gamma \big (\mu
+ \frac{\mu\mu^\gamma}{(\bar\lambda/\mu)^2}\big)\leq C \mu^\gamma\mu,
\end{split}
\end{equation*}
with $C$ depending on $\omega,\gamma, i$ and where we used the bounds in (\ref{ass_H0}) on
the derivatives of $H$ up to $2+2(i-1)$. This ends the proof of Proposition \ref{prop1}.
\end{proof}

\section{The quadruple step construction in two codimensions}\label{sec_stage_prep}

In preparation for the recursive construction of the ``stage'' in the
proof of Theorem \ref{thm_stage}, we first present its main building block, whose bounds 
we index using the eventual recursion counter $j=0\ldots K$. The given quantities are
referred to through the subscript $j$ while the derived quantities
carry the consecutive subscript $j+1$. Namely, we have:

\begin{proposition}\label{prop2}
Let $\omega\subset\R^2$ be open, bounded, smooth and let $N\geq 1$,
$\gamma\in (0,1)$. There exists $l_0\in (0,1)$
depending on $\omega$, and $\sigma_0\geq 1$ depending on $\omega,
\gamma, N$ such that the following holds. Given are the following
quantities:
\begin{itemize}
\item a constant $l>0$, an integer $M$, and frequencies
$\mu_{j-1}, \lambda_j, \mu_j, \lambda_{j+1}, \mu_{j+1}$, satisfying:
\begin{equation}\label{ass_de2}
\begin{split}
& l\leq l_0,\qquad \lambda_{j+1}^{1-\gamma}\geq \mu_j\sigma_0,\qquad M\geq 1, \\
& 1< \mu_{j-1}\leq\lambda_j\leq\mu_j\leq\lambda_{j+1}\leq\mu_{j+1},
\end{split}
\end{equation}
\item two positive auxiliary constants $B_{j}$ and $\tilde C_j$, satisfying:
\begin{equation}\label{ass_ce2}
(B_j/\tilde C_j)^{1/2} \leq\min\Big\{\frac{\mu_j}{\lambda_j}, \frac{\mu_{j+1}}{\lambda_{j+1}}\Big\},
\end{equation}
\item the vector and matrix fields $v_{j}, w_j\in\mathcal{C}^\infty(\bar\omega+\bar
B_{l}(0),\R^{2})$, $A_0\in\mathcal{C}^\infty(\bar\omega+\bar
B_{l}(0),\R^{2\times 2}_\sym)$, such that together with the derived field $\mathcal{D}_{j}
= A_0 - \big(\frac{1}{2} (\nabla v_j)^T\nabla v_j +\sym\nabla
w_j\big)$, they obey:
\begin{align*}
& \begin{array}{ll} 
 \|\nabla^{(m)}\mathcal{D}_j\|_0\leq \tilde C_j\mu_j^m
& \mbox{ for all } \; m=0\ldots M+3N+4, \vspace{2mm}
\end{array} 
\tag*{(\theequation)$_1$}\refstepcounter{equation} \label{ass_HQ2} \\
& \left\{\begin{array}{l}
\|\partial_1^{(t)}\partial_2^{(s)} \partial_{11}v^1_j\|_0\leq
B_{j}^{1/2}\mu_{j-1}^{\gamma/2}\lambda_j^{t+1}\mu_{j-1}^{s},  \qquad 
\|\partial_1^{(t)}\partial_2^{(s)} \partial_{12}v^1_j\|_0\leq
B_{j}^{1/2}\mu_{j-1}^{\gamma/2}\lambda_j^{t}\mu_{j-1}^{s+1}, \vspace{2mm} \\
\|\partial_1^{(t)}\partial_2^{(s)} \partial_{22}v^1_j\|_0\leq
B_{j}^{1/2}\mu_{j-1}^{\gamma/2}\lambda_j^{t-1}\mu_{j-1}^{s+2}, \vspace{2mm}\\
\|\partial_1^{(t)}\partial_2^{(s)} \partial_{11}v^2_j\|_0\leq
B_{j}^{1/2}\mu_{j-1}^{\gamma/2}\lambda_j^{t+2}\mu_{j}^{s-1} \qquad
\|\partial_1^{(t)}\partial_2^{(s)} \partial_{12}v^2_j\|_0\leq
B_{j}^{1/2}\mu_{j-1}^{\gamma/2}\lambda_j^{t+1}\mu_{j}^{s} \vspace{2mm}\\ 
\|\partial_1^{(t)}\partial_2^{(s)} \partial_{22}v^2_j\|_0\leq
B_{j}^{1/2}\mu_{j-1}^{\gamma/2}\lambda_j^{t}\mu_{j}^{s+1} \vspace{1mm}
\end{array} \right. \\ 
& \qquad \qquad \qquad\qquad \mbox{for all } \; t+s=0\ldots M+N+2.
\tag*{(\theequation)$_2$} \label{ass_HQ3}
\end{align*}
\end{itemize}
Then, there exist $v_{j+1}, w_{j+1}\in\mathcal{C}^\infty(\bar\omega+\bar
B_{l}(0),\R^{2})$, such that denoting the new derived field ${\mathcal{D}}_{j+1}
= A_0 - \big(\frac{1}{2} (\nabla v_{j+1})^T\nabla v_{j+1} +\sym\nabla w_{j+1}\big)$,
there hold the estimates:
\begin{itemize}
\item the counterparts of \ref{ass_HQ2} and \ref{ass_HQ3} at the new counter $j+1$:
\begin{align*}
& \begin{array}{l} \displaystyle{\|\nabla^{(m)}{\mathcal{D}}_{j+1}\|_0\leq C\tilde
C_j\mu_{j+1}^m \Big( \frac{1}{(\lambda_{j+1}/\mu_j)^N}
+ \frac{1}{(\mu_{j+1}/\lambda_{j+1})^2} + \frac{1}{(\mu_{j}/\lambda_{j})^2}\Big)} \vspace{2mm}\\
\qquad\qquad\qquad \qquad\qquad  \qquad\qquad 
\qquad\qquad \qquad\qquad \quad 
\mbox{ for all } \; m=0\ldots M, \vspace{2mm}\end{array}
\tag*{(\theequation)$_1$} \label{P2bound4}
\\ & \left\{\begin{array}{l}
\displaystyle{ \|\partial_1^{(t)}\partial_2^{(s)} \partial_{11}v^1_{j+1}\|_0\leq
C \tilde C_j^{1/2} \mu_{j}^{\gamma/2}\lambda_{j+1}^{t+1}\mu_j^s,}\vspace{2mm} \\
\displaystyle{\|\partial_1^{(t)}\partial_2^{(s)} \partial_{12}v^1_{j+1}\|_0\leq
C \tilde C_j^{1/2} \mu_{j}^{\gamma/2}\lambda_{j+1}^{t}\mu_j^{s+1},}\vspace{2mm} 
\\  \displaystyle{\|\partial_1^{(t)}\partial_2^{(s)} \partial_{22}v^1_{j+1}\|_0\leq
C \tilde C_j^{1/2} \mu_{j}^{\gamma/2}\lambda_{j+1}^{t-1}\mu_j^{s+2}
\Big(\frac{1}{(\lambda_j/\mu_{j-1})^2(\mu_{j}/\lambda_{j+1})} +1\Big),}\vspace{2mm}
\\
\displaystyle{\|\partial_1^{(t)}\partial_2^{(s)} \partial_{11}v^2_{j+1}\|_0\leq
C \tilde C_j^{1/2} \mu_{j}^{\gamma/2}\lambda_{j+1}^{t+2}\mu_{j+1}^{s-1}
\Big(\frac{1}{(\lambda_{j+1}/\lambda_j) (\lambda_{j+1}/\mu_{j+1})}+1\Big),}\vspace{2mm} \\ 
\displaystyle{\|\partial_1^{(t)}\partial_2^{(s)} \partial_{12}v^2_{j+1}\|_0\leq
C \tilde C_j^{1/2} \mu_{j}^{\gamma/2}\lambda_{j+1}^{t+1}\mu_{j+1}^{s},}
\vspace{2mm} \\ 
\displaystyle{\|\partial_1^{(t)}\partial_2^{(s)} \partial_{22}v^2_{j+1}\|_0\leq C \tilde C_j^{1/2} \mu_{j}^{\gamma/2}\lambda_{j+1}^{t}\mu_{j+1}^{s+1},
\qquad\qquad   \mbox{ for all } \; t+s=0\ldots M,}
\vspace{1mm} 
\end{array}\right. \tag*{(\theequation)$_2$} \label{P2bound2}
\end{align*}
with constants $C$ depending on $\omega, \gamma, N, M$,
\item bounds on lower derivative orders on displacements in $v$ and $w$:
\begin{align*}
& \begin{array}{l}  \|v_{j+1}- v_j\|_1
\leq C \tilde C_j^{1/2}\mu_j^{\gamma/2},\end{array}\vspace{4mm} 
\tag*{(\theequation)$_1$}\refstepcounter{equation} \label{P2bound1}\\
& \begin{array}{l} \displaystyle{\|w_{j+1}-w_j\|_1\leq C \tilde
    C_j^{1/2}\mu_j^{\gamma/2} \big(\|\nabla v_j\|_0 + 
\tilde C_j^{1/2}\mu_j^{\gamma/2} \big),} \vspace{2mm}\\
\displaystyle{\|\nabla^2(w_{j+1}-w_j)\|_0\leq C \tilde
  C_j^{1/2}\mu_j^{\gamma/2} \big(\|\nabla v_j\|_0 + 
\tilde C_j^{1/2}\mu_j^{\gamma/2} \big)\mu_{j+1},}
\end{array}
\vspace{4mm} \tag*{(\theequation)$_2$}\label{P2bound3} 
\end{align*}
with $C$ depending on $\omega,\gamma, N$.
\end{itemize}
\end{proposition}

\begin{proof}
{\bf 1. (Applying Proposition \ref{prop1} and adding the first corrugation)} 
We apply Proposition \ref{prop1} on the set $\omega$ with parameters $ \gamma, N, l_0$,
$\sigma_0$, and $l, M$, and:
$$\mu=\mu_j,\quad \bar \lambda=\lambda_{j+1},\quad H=\frac{1}{\tilde C_j}\mathcal{D}_j,$$
upon validating (\ref{ass_de}) by (\ref{ass_de2}),
and  (\ref{ass_H0}) by \ref{ass_HQ2}.
Having thus obtained the fields $a, \Psi, \mathcal{F}$, 
we define $a_{j+1}\in \mathcal{C}^\infty(\bar\omega +
\bar B_l(0), \R)$, $\Psi_{j+1}\in\mathcal{C}^\infty(\bar\omega +
\bar B_l(0),\R^2)$, $\mathcal{F}_{j+1}\in
\mathcal{C}^\infty(\bar\omega+\bar B_l(0),\R^{2\times 2}_\sym)$ by: 
\begin{equation*}
a_{j+1} = \tilde C_{j}^{1/2}a, \qquad \Psi_{j+1}=\tilde C_j\Psi,
\qquad \mathcal{F}_{j+1} = \tilde C_j\mathcal{F}.
\end{equation*}
The field $\mathcal{F}_{j+1}$ is the first of the four error fields to
be constructed in the proof, given by:
\begin{equation}\label{def_Fk+1}
\mathcal{F}_{j+1} =  a_{j+1}^2\Id_2 + \sym\nabla \Psi_{j+1} - \mathcal{D}_j
+\frac{1}{\lambda_{j+1}^2}\nabla a_{j+1}\otimes\nabla a_{j+1}.
\end{equation}
Properties (\ref{Ebounds}) yield the following:
\begin{align*}
&  \frac{\bar C\tilde C_j}{2}\mu_j^\gamma\leq a_{j+1}^2\leq
\frac{3\bar C \tilde C_j}{2}\mu_j^\gamma \quad\mbox{ and }\quad
\frac{\bar C^{1/2} \tilde C_j^{1/2}}{2}\mu_j^{\gamma/2}\leq
a_{j+1}\leq \frac{3\bar C^{1/2} \tilde C_j^{1/2}}{2}\mu_j^{\gamma/2}, 
\tag*{(\theequation)$_1$}\refstepcounter{equation} \label{Ebound13}
\vspace{1mm} \\
& \|\nabla^{(m)} a_{j+1}^2\|_0\leq C\tilde
C_j\mu_j^{\gamma}\mu_j^m \; \quad \mbox{ and }\quad \|\nabla^{(m)} a_{j+1}\|_0\leq C\tilde
C_j^{1/2}\mu_j^{\gamma/2}\mu_j^m 
\quad\mbox{ for } \; m=0\ldots M + N +4, \vspace{1mm}
\tag*{(\theequation)$_2$}\label{Ebound23} \\ 
& \displaystyle{\|\nabla^{(m)}\mathcal{F}_{j+1}\|_0 \leq 
C \tilde C_j \frac{\mu_{j}^m}{(\lambda_{j+1}/\mu_j)^N}}
\quad\mbox{ for all } \; m=0\ldots M, 
\tag*{(\theequation)$_3$} \label{Ebound33}\\
& \|\Psi_{j+1}\|_1\leq C\tilde C_j\mu_j^\gamma, \qquad
\|\nabla^2\Psi_{j+1}\|_0\leq C\tilde C_j\mu_j^\gamma\mu_j.
\tag*{(\theequation)$_4$} \label{Ebound43}
\end{align*}
where the constant $\bar C$ depends only on $\omega, \gamma$, while
constants $C$ depend on: in \ref{Ebound43} on $\omega,\gamma,N$ and
in \ref{Ebound23}, \ref{Ebound33} on $\omega,\gamma, N, M$.
We now define the first of the four pairs of the 
intermediate fields $V_1, W_1 \in\mathcal{C}^\infty(\bar\omega +\bar B_{l}(0),\R^2)$,
by setting, in accordance with Lemma \ref{lem_step2} at $i,k=1$:
\begin{equation} \label{VW1}
\begin{split}
& V_1  = v_{j} + \frac{a_{j+1}}{\lambda_{j+1}} \Gamma(\lambda_{j+1} x_1)e_1,\\
& W_1 = w_{j} - \frac{a_{j+1}}{\lambda_{j+1}} \Gamma(\lambda_{j+1} x_1)\nabla v^1_j 
+ \frac{a_{j+1}}{\lambda_{j+1}^2}\bar\Gamma(\lambda_{j+1}x_1)\nabla a_{j+1} 
+ \frac{a_{j+1}^2}{\lambda_{j+1}}\dbar\Gamma(\lambda_{j+1} x_1)e_1 + \Psi_{j+1}.
\end{split}
\end{equation}
Consequently, recalling (\ref{def_Fk+1}) and (\ref{step_err2}) we get:
\begin{equation}\label{err1}
\begin{split}
& \mathcal{D}(V_1, W_1)   \doteq A_0-\big(\frac{1}{2}(\nabla V_1)^T\nabla
V_1+\sym\nabla W_1\big) \\ & = \mathcal{D}_{j}-\Big(\big(\frac{1}{2}(\nabla V_1)^T\nabla
V_1+\sym\nabla W_1\big) - \big(\frac{1}{2}(\nabla v_j)^T\nabla
v_j + \sym\nabla w_j\big) \Big) \\ &
= a_{j+1}^2\mbox{Id}_2 + \sym\nabla \Psi_{j+1} +
 \frac{1}{\lambda_{j+1}^2}\nabla a_{j+1}\otimes \nabla a_{j+1} - \mathcal{F}_{j+1}\\
&\quad - \Big(a_{j+1}^2e_1{\otimes} e_1 -
\frac{a_{j+1}}{\lambda_{j+1}}\Gamma(\lambda_{j+1}x_1) \nabla^2v_j^1 +
\frac{a_{j+1}}{\lambda_{j+1}^2}\bar\Gamma(\lambda_{j+1}x_1)\nabla^2a_{j+1}
\\ & \qquad\qquad \qquad \quad \;
+ \frac{1}{\lambda_{j+1}^2}\tbar\Gamma(\lambda_{j+1}x_1)\nabla
a_{j+1}{\otimes }\nabla a_{j+1} + \sym\nabla\Psi_{j+1}\Big).
\\ & = a_{j+1}^2 e_2\otimes e_2 - \mathcal{F}_{j+1}
+ \Big(\frac{a_{j+1}}{\lambda_{j+1}}\Gamma(\lambda_{j+1}x_1) \nabla^2v_j^1 -
\frac{a_{j+1}}{\lambda_{j+1}^2}\bar\Gamma(\lambda_{j+1}x_1)\nabla^2a_{j+1}
\\ & \qquad\qquad\qquad \qquad\qquad \qquad 
- \frac{1}{\lambda_{j+1}^2}(\tbar\Gamma-1)(\lambda_{j+1}x_1)\nabla
a_{j+1}{\otimes }\nabla a_{j+1} \Big).
\end{split}
\end{equation}

\smallskip

{\bf 2. (Applying Lemma \ref{lem_IBP})} 
We now apply the formula (\ref{ibp1}) with $k=N$ to $H$
being each of the matrix fields $a_{j+1}\nabla v_j^1$ and
$a_{j+1}\nabla^2a_{j+1}$ and $\nabla a_{j+1}\otimes \nabla a_{j+1}$,
to the frequency $\lambda = \lambda_{j+1}$, and to the 
profile $\Gamma_0$ being each of: $\Gamma$, $\bar\Gamma$ and
$\tbar\Gamma-1$, respectively. Define:
\begin{equation} \label{VW2}
\begin{split}
& V_2  = V_{1},\\
& W_2 = W_1 +\sum_{i=0}^N (-1)^i\frac{\Gamma_{i+1}(\lambda_{j+1} x_1)}{\lambda_{j+1}^{i+2}}
L_i\big(a_{j+1}\nabla^2v_j^1\big) 
\\ & \qquad - \sum_{i=0}^N (-1)^i\frac{\bar\Gamma_{i+1}(\lambda_{j+1} x_1)}{\lambda_{j+1}^{i+3}}
L_i\big(a_{j+1}\nabla^2a_{j+1}\big)
\\ & \qquad - \sum_{i=0}^N (-1)^i\frac{(\tbar\Gamma-1)_{i+1}(\lambda_{j+1} x_1)}{\lambda_{j+1}^{i+3}}
L_i\big(\nabla a_{j+1}\otimes \nabla a_{j+1}\big),
\end{split}
\end{equation}
where we recall the definition of $\{L_i\}_{i=0}^{N}$ in
(\ref{coef_ibp1}) and note that the recursive formula for taking the antiderivatives
in (\ref{recu_ibp}) returns periodic functions $\{\Gamma_i, \bar\Gamma_i,
(\tbar\Gamma-1)_i\}_{i=0}^{N+1}$ whose all derivatives are bounded. 
By Lemma \ref{lem_IBP} and (\ref{err1}) we write:
\begin{equation}\label{err2}
\begin{split}
\mathcal{D}(V_2, W_2)   & \doteq A_0-\big(\frac{1}{2}(\nabla V_2)^T\nabla
V_2+\sym\nabla W_2\big) = \mathcal{D}(V_1, W_1)-\sym\nabla (W_2-W_1) \\ &
= - \mathcal{F}_{j+1} +\mathcal{G}_{j+1}
+ \big(a_{j+1}^2 + G\big) e_2\otimes e_2,
\end{split}
\end{equation}
where the second error field $\mathcal{G}_{j+1}\in
\mathcal{C}^\infty(\bar\omega+\bar B_l(0),\R^{2\times 2}_\sym)$ is given by:
\begin{equation}\label{def_Gj+1}
\begin{split}
\mathcal{G}_{j+1}= \; &  (-1)^{N+1}\frac{\Gamma_{N+1}(\lambda_{j+1}x_1)}{\lambda_{j+1}^{N+2}}
\, \sym\nabla L_N\big(a_{j+1}\nabla^2v_j^1\big) 
\\ & - (-1)^{N+1}\frac{\bar\Gamma_{N+1}(\lambda_{j+1} x_1)}{\lambda_{j+1}^{N+3}}
\, \sym\nabla L_N\big(a_{j+1}\nabla^2a_{j+1}\big)
\\ & - (-1)^{N+1}\frac{(\tbar\Gamma-1)_{N+1}(\lambda_{j+1} x_1)}{\lambda_{j+1}^{N+3}}
\, \sym\nabla L_N\big(\nabla a_{j+1}\otimes \nabla a_{j+1}\big),
\end{split}
\end{equation}
together with the corrector function $G\in\mathcal{C}^\infty(\bar\omega +\bar B_l(0),\R)$ in:
\begin{equation}\label{def_G}
\begin{split}
G = \; & \sum_{i=0}^N (-1)^i\frac{\Gamma_{i}(\lambda_{j+1} x_1)}{\lambda_{j+1}^{i+1}}
P_i\big(a_{j+1}\nabla^2v_j^1\big) 
- \sum_{i=0}^N (-1)^i\frac{\bar\Gamma_{i}(\lambda_{j+1} x_1)}{\lambda_{j+1}^{i+2}}
P_i\big(a_{j+1}\nabla^2a_{j+1}\big)
\\ & - \sum_{i=0}^N (-1)^i\frac{(\tbar\Gamma-1)_{i}(\lambda_{j+1} x_1)}{\lambda_{j+1}^{i+2}}
P_i\big(\nabla a_{j+1}\otimes \nabla a_{j+1}\big)
\end{split}
\end{equation}
where $P_i$ are given by the formula (\ref{coef_ibp1}).
In the next two steps of our proof, we will show that:
\begin{equation}\label{bound_Gj+1}
\|\nabla^{(m)}\mathcal{G}_{j+1}\|_0\leq C \tilde C_j 
\frac{\lambda_{j+1}^m}{(\lambda_{j+1}/\mu_j)^{N}} 
\qquad \mbox{for all } \; m=0\ldots M,
\end{equation}
which is similar to the bound on the first error term
$\mathcal{F}_{j+1}$ in \ref{Ebound33}, and that:
\begin{equation}\label{bound_G}
\|\nabla^{(m)}G\|_0\leq C \tilde C_j
\frac{\mu_j^{\gamma}}{\lambda_{j+1}/\mu_j} \lambda_{j+1}^m.
\qquad \mbox{for all } \; m=0\ldots M+2,
\end{equation}
%\begin{equation}\label{bound_G}
%\|\partial_1^{(t)}\partial_2^{(s)}G\|_0\leq C \tilde C_j
%\lambda_{j+1}^t\mu_j^s \frac{\mu_j^{\gamma}}{\lambda_{j+1}/\mu_j}.
%\qquad \mbox{for all } \; t+s=0\ldots M,
%\end{equation}
which is lower order than the bound on $a_{j+1}^2$ in \ref{Ebound23}.
The constants $C$ depend on $\omega, \gamma, N, M$. 

\smallskip

{\bf 3. (Proof of the estimate (\ref{bound_Gj+1}))} 
By (\ref{coef_ibp1}) we observe that the right hand side of
(\ref{def_Gj+1}) consists of three types of terms:
\begin{equation*}
\begin{split}
& I_{N+1} = \frac{\Gamma_{N+1}(\lambda_{j+1}x_1)}{\lambda_{j+1}^{N+2}}
\nabla^{(N+1)}\big(a_{j+1}\nabla^2v_j^1\big), \qquad 
II_{N+1} = \frac{\bar\Gamma_{N+1}(\lambda_{j+1}x_1)}{\lambda_{j+1}^{N+3}}
\nabla^{(N+1)}\big(a_{j+1}\nabla^2 a_{j+1}\big),
\\ & III_{N+1} = \frac{(\tbar\Gamma-1)_{N+1}(\lambda_{j+1}x_1)}{\lambda_{j+1}^{N+3}}
\nabla^{(N+1)}\big(\nabla a_{j+1}\otimes \nabla a_{j+1}\big).
\end{split}
\end{equation*}
Using the bounds for $\nabla^2v_j^1$ in \ref{ass_HQ3} and for
$a_{j+1}$ in \ref{Ebound23}, we obtain:
\begin{equation*}
\begin{split}
& \|\nabla^{(m)}I_{N+1}\|_0  \leq  C\sum_{p+q=m} \lambda_{j+1}^{p-N-2}\|\nabla^{(q+N+1)}
\big(a_{j+1}\nabla^2v_j^1\big)\|_0 \\ & \leq C \hspace{-0.8cm}
\sum_{\tiny\begin{array}{c} p+q=m\\  u+z = q+N+1\end{array}} \hspace{-0.8cm}
\lambda_{j+1}^{p-N-2}\|\nabla^{(u)} a_{j+1}\|_0\|\nabla^{(z+2)}v_j^1\|_0 
\leq C \tilde C_j^{1/2}B_j^{1/2}\mu_j^{\gamma}
\hspace{-0.8cm} \sum_{\tiny\begin{array}{c} p+q=m\\  u+z = q+N+1\end{array}} \hspace{-0.8cm}
\lambda_{j+1}^{p-N-2} \mu_j^u \lambda_j^{z+1}
\\ & \leq  C \tilde C_j^{1/2}B_j^{1/2}\mu_j^{\gamma} \sum_{p+q=m} \lambda_{j+1}^{p-N-2}
\mu_j^{q+N+1}\lambda_j \leq C \tilde C_j^{1/2}B_j^{1/2}\mu_j^{\gamma} 
\frac{\lambda_{j+1}^m}{(\lambda_{j+1}/\mu_j)^{N+1}
  (\lambda_{j+1}/\lambda_j)} \\ & \leq C \tilde C_j \lambda_{j+1}^m
\frac{(B_j/\tilde C_j)^{1/2}}{(\lambda_{j+1}/\mu_j)^{N}
  (\lambda_{j+1}/\lambda_j)} 
\end{split}
\end{equation*}
where the last inequality follows by the second assumption in
(\ref{ass_de2}) and with the constant $C$ depending on $\omega,
\gamma, N, M$. Similarly, we get:
\begin{equation*}
\begin{split}
& \|\nabla^{(m)}II_{N+1}\|_0 + \|\nabla^{(m)}III_{N+1}\|_0  
\\ & \leq C \hspace{-0.8cm}
\sum_{\tiny\begin{array}{c} p+q=m\\  u+z = q+N+1\end{array}} \hspace{-0.8cm}
\lambda_{j+1}^{p-N-3}\big(\|\nabla^{(u)} a_{j+1}\|_0\|\nabla^{(z+2)}a_{j+1}\|_0 
+ \|\nabla^{(u+1)} a_{j+1}\|_0\|\nabla^{(z+1)}a_{j+1}\|_0 \big)
\\ & \leq C \tilde C_j \mu_j^{\gamma}
\hspace{-0.8cm} \sum_{\tiny\begin{array}{c} p+q=m\\  u+z = q+N+1\end{array}} \hspace{-0.8cm}
\lambda_{j+1}^{p-N-3} \mu_j^{u+z+2} \leq  C \tilde C_j\mu_j^{\gamma} \sum_{p+q=m} \lambda_{j+1}^{p-N-3}
\mu_j^{q+N+3}\\ & \leq C \tilde C_j\mu_j^{\gamma} 
\frac{\lambda_{j+1}^m}{(\lambda_{j+1}/\mu_j)^{N+3}} \leq C \tilde C_j 
\frac{\lambda_{j+1}^m}{(\lambda_{j+1}/\mu_j)^{N+2}}.
\end{split}
\end{equation*}
In conclusion, (\ref{bound_Gj+1}) follow in view of (\ref{ass_ce2})
and above obtained bound:
$$\|\nabla^{(m)}\mathcal{G}_{j+1}\|_0 \leq  C \tilde C_j 
\frac{\lambda_{j+1}^m}{(\lambda_{j+1}/\mu_j)^{N}} 
\Big(\frac{(B_j/\tilde C_j)^{1/2}}{\lambda_{j+1}/\lambda_j} +1\Big).$$

\smallskip

{\bf 4. (Proof of the estimate (\ref{bound_G}))} 
As before, formulas (\ref{coef_ibp1}) imply that the right hand side of
(\ref{def_G}) consists of the following types of terms:
\begin{equation*}
\begin{split}
& IV_0= \frac{\Gamma(\lambda_{j+1}x_1)}{\lambda_{j+1}}
\,a_{j+1}\partial_{22}v_j^1,\qquad
IV_{i} = \frac{\Gamma_{i}(\lambda_{j+1}x_1)}{\lambda_{j+1}^{i+1}}
\partial_1^{(i-1)}\partial_2\big(a_{j+1}\partial_{12}v_j^1\big) 
\quad \mbox{ for } \; i= 1\ldots N, \\ & 
V_{i} = \frac{\Gamma_{i}(\lambda_{j+1}x_1)}{\lambda_{j+1}^{i+1}}
\partial_1^{(i-2)}\partial_2^{(2)}\big(a_{j+1}\partial_{11}v_j^1\big) 
\quad \mbox{ for }\; i= 2\ldots N,
\\ & II_{i} = \frac{\bar\Gamma_{i}(\lambda_{j+1}x_1)}{\lambda_{j+1}^{i+2}}
\nabla^{(i)}\big(a_{j+1}\nabla^2 a_{j+1}\big),
\quad  III_{i} = \frac{(\tbar\Gamma-1)_{i}(\lambda_{j+1}x_1)}{\lambda_{j+1}^{i+2}}
\nabla^{(i)}\big(\nabla a_{j+1}^{\otimes 2}\big)\quad  \mbox{ for } \;i= 0\ldots N.
\end{split}
\end{equation*}
By exactly the same argument as in the previous step, we obtain the bound:
\begin{equation*}
\|\nabla^{(m)}II_{i}\|_0 + \|\nabla^{(m)}III_{i}\|_0  
\leq C \tilde C_j \frac{\lambda_{j+1}^m}{(\lambda_{j+1}/\mu_j)^{i+1}}
\quad\mbox{ for all }\; i=0\ldots N.
\end{equation*}
The finer bounds on $IV_0$, $IV_i$ and $V_i$ follow from \ref{ass_HQ2} and \ref{Ebound23}:
\begin{equation*}
\begin{split}
& \|\partial_1^{(t)}\partial_2^{(s)}IV_0\|_0  
\leq  C\sum_{p_1+q_1=t} \lambda_{j+1}^{p_1-1}\|\partial_1^{(q_1)}\partial_2^{(s)}
\big(a_{j+1}\partial_{22}v_j^1\big)\|_0 \\ & 
\qquad\qquad\qquad \leq C \hspace{-0.5cm}
\sum_{\tiny\begin{array}{c} p_1+q_1=t \\  u_1+z_1 = q_1 \\ u_2+z_2=s \end{array}} \hspace{-0.5cm}
\lambda_{j+1}^{p_1-1}\|\partial_1^{(u_1)}\partial_2^{(u_2)} a_{j+1}\|_0
\|\partial_1^{(z_1)}\partial_2^{(z_2+2)} v_j^1\|_0 
\\ & \leq C \tilde C_j^{1/2}B_j^{1/2}\mu_j^{\gamma} \hspace{-0.5cm}
\sum_{\tiny\begin{array}{c} p_1+q_1=t \\  u_1+z_1 = q_1 \\ u_2+z_2=s \end{array}} \hspace{-0.5cm}
\lambda_{j+1}^{p_1-1}\mu_j^{u_1+u_2}\lambda_j^{z_1-1}\mu_{j-1}^{z_2+2}
\leq  C \tilde C_j^{1/2}B_j^{1/2}\mu_j^{\gamma} \sum_{p_1+q_1=t} \lambda_{j+1}^{p_1-1}
\mu_j^{q_1+s}\lambda_j^{-1} \mu_{j-1}^2 \\ & 
\leq C \tilde C_j^{1/2}B_j^{1/2}\mu_j^{\gamma} 
\frac{\lambda_{j+1}^t\mu_j^s}{\lambda_{j+1}\lambda_j/\mu_{j-1}^2} 
= C \tilde C_j \lambda_{j+1}^t\mu_j^s\mu_j^\gamma
\frac{(B_j/\tilde C_j)^{1/2}}{(\mu_{j}/\lambda_j)(\lambda_j/\mu_{j-1})^2}
\cdot \frac{1}{\lambda_{j+1}/\mu_j}.
\end{split}
\end{equation*}
Similarly, for all $i=1\ldots N$ and with $C$ depending on $\omega,
\gamma, N, M$, we get:
\begin{equation*}
\begin{split}
& \|\partial_1^{(t)}\partial_2^{(s)}IV_i\|_0  
\leq  C\sum_{p_1+q_1=t} \lambda_{j+1}^{p_1-i-1}\|\partial_1^{(q_1+i-1)}\partial_2^{(s+1)}
\big(a_{j+1}\partial_{12}v_j^1\big)\|_0 \\ & 
\qquad\qquad\qquad \leq C \hspace{-0.6cm}
\sum_{\tiny\begin{array}{c} p_1+q_1=t \\  u_1+z_1 = q_1 +i-1 \\ u_2+z_2=s+1 \end{array}} \hspace{-0.7cm}
\lambda_{j+1}^{p_1-i-1}\|\partial_1^{(u_1)}\partial_2^{(u_2)} a_{j+1}\|_0
\|\partial_1^{(z_1+1)}\partial_2^{(z_2+1)} v_j^1\|_0 
\\ & \leq C \tilde C_j^{1/2}B_j^{1/2}\mu_j^{\gamma} \hspace{-0.6cm}
\sum_{\tiny\begin{array}{c} p_1+q_1=t \\  u_1+z_1 = q_1 +i-1 \\ u_2+z_2=s+1 \end{array}} \hspace{-0.7cm}
\lambda_{j+1}^{p_1-i-1}\mu_j^{u_1+u_2}\lambda_j^{z_1} \mu_{j-1}^{z_2-1} 
\leq  C \tilde C_j^{1/2}B_j^{1/2}\mu_j^{\gamma} \sum_{p_1+q_1=t} \lambda_{j+1}^{p_1-i-1}
\mu_j^{q_1+i+ s}\mu_{j-1}^{-1} \\ & 
\leq C \tilde C_j^{1/2}B_j^{1/2}\mu_j^{\gamma} 
\frac{\lambda_{j+1}^t\mu_j^s}{(\lambda_{j+1}/\mu_j)^i(\lambda_{j+1}/\mu_{j-1})} 
\leq C \tilde C_j^{1/2}B_j^{1/2}\mu_j^{\gamma} 
\frac{\lambda_{j+1}^t\mu_j^s}{(\lambda_{j+1}/\mu_j)^2(\mu_{j}/\mu_{j-1})} 
\\ & = C \tilde C_j \lambda_{j+1}^t\mu_j^s\mu_j^\gamma
\frac{(B_j/\tilde C_j)^{1/2}}{(\mu_{j}/\lambda_j)(\lambda_{j+1}/\mu_j)(\lambda_j/\mu_{j-1})}
\cdot \frac{1}{\lambda_{j+1}/\mu_j},
\end{split}
\end{equation*}
and also, for all $i=2\ldots N$:
\begin{equation*}
\begin{split}
& \|\partial_1^{(t)}\partial_2^{(s)}V_i\|_0  
\leq  C\sum_{p_1+q_1=t} \lambda_{j+1}^{p_1-i-1}\|\partial_1^{(q_1+i-2)}\partial_2^{(s+2)}
\big(a_{j+1}\partial_{11}v_j^1\big)\|_0 \\ & 
\qquad\qquad\quad \; \, \leq C \hspace{-0.6cm}
\sum_{\tiny\begin{array}{c} p_1+q_1=t \\  u_1+z_1 = q_1 +i-2 \\ u_2+z_2=s+2 \end{array}} \hspace{-0.7cm}
\lambda_{j+1}^{p_1-i-1}\|\partial_1^{(u_1)}\partial_2^{(u_2)} a_{j+1}\|_0
\|\partial_1^{(z_1+2)}\partial_2^{(z_2)} v_j^1\|_0 
\\ & \leq C \tilde C_j^{1/2}B_j^{1/2}\mu_j^{\gamma} \hspace{-0.6cm}
\sum_{\tiny\begin{array}{c} p_1+q_1=t \\  u_1+z_1 = q_1 +i-2 \\ u_2+z_2=s+2 \end{array}} \hspace{-0.7cm}
\lambda_{j+1}^{p_1-i-1}\mu_j^{u_1+u_2}\lambda_j^{z_1+1} \mu_{j-1}^{z_2} 
\leq  C \tilde C_j^{1/2}B_j^{1/2}\mu_j^{\gamma} \sum_{p_1+q_1=t} \lambda_{j+1}^{p_1-i-1}
\mu_j^{q_1+i+ s}\lambda_j\\ & 
\leq C \tilde C_j^{1/2}B_j^{1/2}\mu_j^{\gamma} 
\frac{\lambda_{j+1}^t\mu_j^s}{(\lambda_{j+1}/\mu_j)^i(\lambda_{j+1}/\lambda_{j})} 
\leq C \tilde C_j^{1/2}B_j^{1/2}\mu_j^{\gamma} 
\frac{\lambda_{j+1}^t\mu_j^s}{(\lambda_{j+1}/\mu_j)^2(\lambda_{j+1}/\lambda_{j})}, 
\end{split}
\end{equation*}
which may be rewritten as:
$$\|\partial_1^{(t)}\partial_2^{(s)}V_i\|_0  
\leq C \tilde C_j \lambda_{j+1}^t\mu_j^s\mu_j^\gamma
\frac{(B_j/\tilde C_j)^{1/2}}{(\mu_{j}/\lambda_j)(\lambda_{j+1}/\mu_j)^2}
\cdot \frac{1}{\lambda_{j+1}/\mu_j}.$$
Gathering the above displayed formulas, we arrive at:
$$\|\partial_1^{(t)}\partial_2^{(s)}G\|_0 \leq  C \tilde C_j
{\lambda_{j+1}^t\mu_j^s \mu_j^\gamma}
\frac{(B_j/\tilde C_j)^{1/2}}{(\mu_{j}/\lambda_j)}\cdot \frac{1}{\lambda_{j+1}/\mu_j},$$
obtaining the bound (\ref{bound_G}) in view of the assumptions (\ref{ass_ce2}).

\smallskip

{\bf 5. (Adding the second corrugation)} We define the third pair of intermediate
fields $V_3, W_3\in\mathcal{C}^\infty(\bar\omega +\bar B_{l}(0),\R^2)$,
using (\ref{defi_per2}) with $i, k=2$ and the amplitude dictated by (\ref{err2}):
\begin{equation}\label{VW3}
\begin{split}
& V_3 = V_2 + \frac{b_{j+1}}{\mu_{j+1}} \Gamma(\mu_{j+1} x_2)e_2, 
\\ & W_3 = W_2 - \frac{b_{j+1}}{\mu_{j+1}} \Gamma(\mu_{j+1} x_2)\nabla v^2_j
+ \frac{b_{j+1}}{\mu_{j+1}^2}\bar\Gamma(\mu_{j+1}x_2)\nabla b_{j+1}+
\frac{b_{j+1}^2}{\mu_{j+1}}\dbar\Gamma(\mu_{j+1} x_2)e_2,
\\ & \mbox{where }\; 
b_{j+1}^2 = a_{j+1}^2 + G.
\end{split}
\end{equation}
Firstly, we argue that $b_{j+1}\in\mathcal{C}^\infty(\bar\omega + \bar B_{l}(0), \R)$ is well
defined and it satisfies:
\begin{align*}
&  \frac{\bar C\tilde C_j}{4}\mu_j^\gamma\leq b_{j+1}^2\leq
{2\bar C \tilde C_j}\mu_j^\gamma \quad\mbox{ and so }\quad
\|b_{j+1}\|_0\leq C\tilde C_j^{1/2}\mu_j^{\gamma/2}, 
\tag*{(\theequation)$_1$}\refstepcounter{equation} \label{b1}\\
&  \|\nabla^{(m)} b_{j+1}\|_0\leq C\tilde
C_j^{1/2}\mu_j^{\gamma/2}\frac{\lambda_{j+1}^m}{\lambda_{j+1}/\mu_j} 
\quad\mbox{ for all } \; m=1\ldots M+2, \vspace{3mm} \tag*{(\theequation)$_2$}\label{b2} 
\end{align*}
where $C$ in \ref{b1} depends on $\omega,\gamma, N$, and in \ref{b2} on
$\omega,\gamma, N, M$. Indeed, by (\ref{bound_G}) we have:
$$\|G\|_0\leq \tilde C_j \mu_j^\gamma\frac{C}{\lambda_{j+1}/\mu_j},$$
so taking $\lambda_{j+1}/\mu_j$ sufficiently large by assigning 
large $\sigma_0$ in the second condition of (\ref{ass_de2}), in
function of the pre-fixed $\omega, \gamma, N$, there follows \ref{b1}
through \ref{Ebound13}. From the same estimates:
$$\|\nabla^{(m)} b_{j+1}^2\|_0\leq C\tilde C_j\mu_j^\gamma 
\frac{\lambda_{j+1}^m}{\lambda_{j+1}/\mu_j} \quad\mbox{ for all } \; m=1\ldots M+2,$$
which implies \ref{b2} by an application of the Fa\'a di Bruno formula:
\begin{equation*}
\begin{split}
\|\nabla^{(m)}b_{j+1}\|_0 & \leq C\Big\|\sum_{p_1+2p_2+\ldots
  mp_m=m} b_{j+1}^{2(1/2-p_1-\ldots -p_m)}\prod_{z=1}^m\big|\nabla^{(z)}b_{j+1}^2\big|^{p_z}\Big\|_0
\\ & \leq C\|b_{j+1}\|_0 \sum_{p_1+2p_2+\ldots
  mp_s=m}\prod_{z=1}^m\Big(\frac{\|\nabla^{(z)}b_{j+1}^2\|_0}{
  \tilde C \tilde C_j \mu^\gamma}\Big)^{p_z}
\leq C \tilde C_j^{1/2}\mu_j^{\gamma/2}\frac{\lambda_{j+1}^m}{\lambda_{j+1}/\mu_j}.
\end{split}
\end{equation*}

Secondly, recalling (\ref{err2}) and Lemma \ref{lem_step2}, we note that:
\begin{equation}\label{err3}
\begin{split}
& \mathcal{D}(V_3, W_3)   \doteq A_0-\big(\frac{1}{2}(\nabla V_3)^T\nabla
V_3+\sym\nabla W_3\big) \\ & = \mathcal{D}(V_2, W_2)  
-\Big(\big(\frac{1}{2}(\nabla V_3)^T\nabla
V_3+\sym\nabla W_3\big) - \big(\frac{1}{2}(\nabla V_2)^T\nabla
V_2 + \sym\nabla W_2\big) \Big) \\ &
= -\mathcal{F}_{j+1}+\mathcal{G}_{j+1} 
+ \frac{b_{j+1}}{\mu_{j+1}}\Gamma(\mu_{j+1}x_2) \nabla^2v_j^2 -\mathcal{H}_{j+1},
\end{split}
\end{equation}
with the third error field $\mathcal{H}_{j+1}\in
\mathcal{C}^\infty(\bar\omega+ \bar B_l(0), \R^{2\times 2}_\sym)$  given by:
\begin{equation}\label{def_Hj+1}
\mathcal{H}_{j+1} =
\frac{b_{j+1}}{\mu_{j+1}^2}\bar\Gamma(\mu_{j+1}x_2)\nabla^2b_{j+1}
+ \frac{1}{\mu_{j+1}^2}\tbar\Gamma(\mu_{j+1}x_2)\nabla
b_{j+1}{\otimes }\nabla b_{j+1}.
\end{equation}
We estimate, using \ref{b1} and \ref{b2}:
\begin{equation}\label{bound_Hj+1}
\begin{split}
& \|\nabla^{(m)}\mathcal{H}_{j+1}\|_0
\leq C\sum_{p+q=m} \mu_{j+1}^{p-2}\Big(\|\nabla^{(q)}\big(b_{j+1}\nabla^2b_{j+1}\big)\|_0
+ \|\nabla^{(q)}\big(\nabla b_{j+1}\otimes \nabla b_{j+1}\big)\|_0
\\ &  \leq C \hspace{-0.5cm}
\sum_{\tiny\begin{array}{c} p+q=m\\  u+z = q\end{array}} \hspace{-0.5cm}
\mu_{j+1}^{p-2}\big(\|\nabla^{(u)} b_{j+1}\|_0\|\nabla^{(z+2)}b_{j+1}\|_0 
+ \|\nabla^{(u+1)} b_{j+1}\|_0\|\nabla^{(z+1)}b_{j+1}\|_0 \big)
\\ & \leq C \tilde C_j \mu_j^{\gamma} \sum_{p+q=m}
\mu_{j+1}^{p-2} \frac{\lambda_{j+1}^{t+q+2} }{\lambda_{j+1}/\mu_j}
\leq C \tilde C_j \frac{\mu_{j+1}^m}{(\mu_{j+1}/\lambda_{j+1})^2}
\cdot \frac{\mu_j^\gamma}{\lambda_{j+1}/\mu_j}
\\ & \leq C \tilde C_j \frac{\mu_{j+1}^m}{(\mu_{j+1}/\lambda_{j+1})^2}
\quad\mbox{ for all }\; m=0\ldots M,
\end{split}
\end{equation}
with $C$ depending on $\omega, \gamma, N, M$.

\smallskip

{\bf 6. (Applying Corollary \ref{cor_IBP})} 
We now apply (\ref{ibp2}) with $k=1$ to $H=b_{j+1}\nabla^2v_j^2$
the frequency $\lambda = \mu_{j+1}$, and the 
profile $\Gamma_0=\Gamma$, and define the final fields:
\begin{equation} \label{VW4}
\begin{split}
& v_{j+1}  = V_{3},\\
& w_{j+1} = W_3 +\sum_{i=0}^1 (-1)^i\frac{\Gamma_{i+1}(\mu_{j+1} x_2)}{\mu_{j+1}^{i+2}}
\tilde L_i\big(b_{j+1}\nabla^2v_j^2\big),
\end{split}
\end{equation}
where we recall the definition of $\{\tilde L_i\}_{i=0}^{1}$ in
(\ref{coef_ibp2}) and the recursive formula of taking the antiderivatives
in (\ref{recu_ibp}). By Corollary \ref{cor_IBP} and (\ref{err3}) we obtain:
\begin{equation}\label{err4}
\mathcal{D}_{j+1} = \mathcal{D}(V_3, W_3)-\sym\nabla (w_{j+1}-W_3) 
= - \mathcal{F}_{j+1} +\mathcal{G}_{j+1} -\mathcal{H}_{j+1} 
+\mathcal{I}_{j+1} 
\end{equation}
where the final, fourth error field $\mathcal{I}_{j+1}\in
\mathcal{C}^\infty(\bar\omega+\bar B_l(0),\R^{2\times 2}_\sym)$ is given by:
\begin{equation}\label{def_Ij+1}
\mathcal{I}_{j+1}= \frac{\Gamma_{2}(\mu_{j+1}x_2)}{\mu_{j+1}^{3}}
\, \sym\nabla \tilde L_1\big(b_{j+1}\nabla^2v_j^2\big) 
+\sum_{i=0}^1 (-1)^{i}\frac{\Gamma_{i}(\mu_{j+1} x_2)}{\mu_{j+1}^{i+1}}
\tilde P_i\big(b_{j+1}\nabla^2v_j^2\big) e_1\otimes e_1,
\end{equation}
and can be estimated as follows:
\begin{equation*}
\begin{split}
& \|\partial_1^{(t)}\partial_2^{(s)}\mathcal{I}_{j+1}\|_0 \leq 
C\sum_{p+q=s}\mu_{j+1}^{p-3}\|\nabla^{(t+q+2)}\big(b_{j+1}\nabla^2v_j^2\big)\|_0
\\ & \qquad \qquad\qquad\qquad + C\hspace{-1mm}
\sum_{p+q=s}\mu_{j+1}^{p-1}\|\partial_1^{(t)} \partial_2^{(q)}\big(b_{j+1}\partial_{11}v_j^2\big)\|_0 
+ C \hspace{-1mm} \sum_{p+q=s}\mu_{j+1}^{p-2}
\|\partial_1^{(t+1)} \partial_2^{(q)}\big(b_{j+1}\partial_{12}v_j^2\big)\|_0
\\ & \leq C \hspace{-6mm} \sum_{\tiny\begin{array}{c} p+q=s \\  u+z = t+q+2  \end{array}} 
\hspace{-6mm} \mu_{j+1}^{p-3}\|\nabla^{(u)} b_{j+1}\|_0\|\nabla^{(z+2)}v_j^2\big)\|_0
+ C\hspace{-4mm} \sum_{\tiny\begin{array}{c} p+q=s \\  u_1+z_1 =  t \\  u_2+z_2 = q \end{array}} 
\hspace{-4mm}  \mu_{j+1}^{p-1}\|\nabla^{(u_1+u_2)} b_{j+1}\|_0\|\partial_1^{(z_1+2)}\partial_2^{(z_2)}v_j^2\|_0 
\\ & \quad + C\hspace{-4mm} \sum_{\tiny\begin{array}{c} p+q=s \\  u_1+z_1 =  t+1 \\  u_2+z_2 = q \end{array}} 
\hspace{-4mm} \mu_{j+1}^{p-2}\|\nabla^{(u_1+u_2)}
b_{j+1}\|_0\|\partial_1^{(z_1+1)}\partial_2^{(z_2+1)}v_j^2\|_0  
\end{split}
\end{equation*}
where using the bounds on $\nabla^2v_j^2$
in \ref{ass_HQ3} together with \ref{b2}, implies:
\begin{equation*}
\begin{split}
& \|\partial_1^{(t)}\partial_2^{(s)}\mathcal{I}_{j+1}\|_0
\leq C\tilde C_j^{1/2}B_j^{1/2}\mu_j^\gamma \Big(\hspace{-4mm}
\sum_{\tiny\begin{array}{c} p+q=s \\  u+z = t+q+2  \end{array}}
\hspace{-6mm} \mu_{j+1}^{p-3}\lambda_{j+1}^u\mu_j^{z+1}
\\ & \qquad \qquad\qquad \qquad\qquad \hspace{-4mm}
+ \hspace{-4mm} \sum_{\tiny\begin{array}{c} p+q=s \\  u_1+z_1 =  t \\  u_2+z_2 = q \end{array}} 
\hspace{-4mm} \mu_{j+1}^{p-1}\lambda_{j+1}^{u_1+u_2}\lambda_{j}^{z_1+2}\mu_j^{z_2-1} + \hspace{-4mm} \sum_{\tiny\begin{array}{c} p+q=s \\  u_1+z_1 =  t+1 \\  u_2+z_2 = q \end{array}} 
\hspace{-4mm} \mu_{j+1}^{p-2}\lambda_{j+1}^{u_1+u_2} \lambda_{j}^{z_1+1}\mu_j^{z_2}\Big)
\\ & \leq C\tilde C_j^{1/2}B_j^{1/2}\mu_j^\gamma \Big(
\sum_{p+q=s} \mu_{j+1}^{p-3}\lambda_{j+1}^{t+q+2}\mu_j + 
\sum_{p+q=s} \mu_{j+1}^{p-1}\lambda_{j+1}^{t+q}\lambda_j^2\mu_j^{-1}
+ \sum_{p+q=s} \mu_{j+1}^{p-2}\lambda_{j+1}^{t+q+1} \lambda_{j}\Big)
\\ & \leq  C\tilde C_j^{1/2}B_j^{1/2}
\frac{\lambda_{j+1}^t\mu_{j+1}^s\mu_j^\gamma}{\lambda_{j+1}/\mu_j}
\Big(\frac{1}{(\mu_{j+1}/\lambda_{j+1})^3} + \frac{1}{(\mu_{j+1}/\lambda_{j+1})(\mu_j/\lambda_{j})^2} 
+ \frac{1}{(\mu_{j+1}/\lambda_{j+1})^2(\mu_j/\lambda_j)} \Big).
\end{split}
\end{equation*}

In conclusion, the second assumption in (\ref{ass_de2}) and (\ref{ass_ce2}) yield:
\begin{equation}\label{bound_Ij+1}
\|\nabla^{(m)}\mathcal{I}_{j+1}\|_0 \leq 
C\tilde C_j \mu_{j+1}^m\Big(
\frac{1}{(\mu_{j+1}/\lambda_{j+1})^2} + \frac{1}{(\mu_{j}/\lambda_{j})^2}\Big)
\qquad\mbox{ for all }\; m=0\ldots M.
\end{equation}

\smallskip

{\bf 7. (Proofs of \ref{P2bound4} -- \ref{P2bound1}) } We deduce \ref{P2bound4} directly,
by the decomposition (\ref{err4}), in view of the four error estimates  \ref{Ebound33},
(\ref{bound_Gj+1}),  (\ref{bound_Hj+1}) and (\ref{bound_Ij+1}).
To show \ref{P2bound2} note that:
\begin{equation*}
\begin{split}
& \|\partial_1^{(t)}\partial_2^{(s)}\big(v_{j+1}^1-v_j^1\big)\|_0\leq C
\sum_{p+q=t}\lambda_{j+1}^{p-1}\|\nabla^{(q+s)} a_{j+1}\|_0 \leq C \tilde
C_j^{1/2}\mu_j^{\gamma/2} \sum_{p+q=t}\lambda_{j+1}^{p-1}\mu_j^{q+s}
\\ & \qquad\qquad \qquad\qquad \quad \; \; 
\leq C C_j^{1/2}\mu_j^{\gamma/2} \lambda_{j+1}^{t-1}\mu_j^{s},
\\ & \|\partial_1^{(t)}\partial_2^{(s)}\big(v_{j+1}^2-v_j^2\big)\|_0\leq C
\sum_{p+q=s}\mu_{j+1}^{p-1}\|\nabla^{(q+t)} b_{j+1}\|_0 \leq C \tilde
C_j^{1/2}\mu_j^{\gamma/2} \sum_{p+q=s}\mu_{j+1}^{p-1}\lambda_{j+1}^{q+t}
\\ & \qquad\qquad \qquad\qquad \quad \; \; 
\leq C C_j^{1/2}\mu_j^{\gamma/2} \lambda_{j+1}^{t}\mu_{j+1}^{s-1},
\end{split}
\end{equation*}
in virtue of definitions (\ref{VW1}), (\ref{VW3}) and the bounds
\ref{Ebound23}, \ref{b1}, \ref{b2}. In particular, the above also
implies \ref{P2bound1}. Further, combined with the assumption
\ref{ass_HQ3}, we get, for all $t+s=0\ldots M$ and with $C$
depending on $\omega, \gamma, N, M$:
\begin{equation*}
\begin{split}
& \|\partial_1^{(t)}\partial_2^{(s)} \partial_{11}v^1_{j+1}\|_0\leq
C \tilde C_j^{1/2} \mu_{j}^{\gamma/2}\lambda_{j+1}^{t+1}\mu_j^s
\Big(\frac{(B_j/\tilde C_j)^{1/2}}{\lambda_{j+1}/\lambda_j} +1\Big),
\\ & \|\partial_1^{(t)}\partial_2^{(s)} \partial_{12}v^1_{j+1}\|_0\leq
C \tilde C_j^{1/2} \mu_{j}^{\gamma/2}\lambda_{j+1}^{t}\mu_j^{s+1}
\Big(\frac{(B_j/\tilde C_j)^{1/2}}{\mu_{j}/\mu_{j-1}} +1\Big),
\\ &   \|\partial_1^{(t)}\partial_2^{(s)} \partial_{22}v^1_{j+1}\|_0\leq
C \tilde C_j^{1/2} \mu_{j}^{\gamma/2}\lambda_{j+1}^{t-1}\mu_j^{s+2}
\Big(\frac{(B_j/\tilde C_j)^{1/2}}{(\mu_j/\mu_{j-1})^2(\lambda_{j}/\lambda_{j+1})} +1\Big),
\\ & 
\|\partial_1^{(t)}\partial_2^{(s)} \partial_{11}v^2_{j+1}\|_0\leq
C \tilde C_j^{1/2} \mu_{j}^{\gamma/2}\lambda_{j+1}^{t+2}\mu_{j+1}^{s-1}
\Big(\frac{(B_j/\tilde C_j)^{1/2}}{(\lambda_{j+1}/\lambda_j)^2(\mu_j/\mu_{j+1})} +1\Big),
\\ & \|\partial_1^{(t)}\partial_2^{(s)} \partial_{12}v^2_{j+1}\|_0\leq
C \tilde C_j^{1/2} \mu_{j}^{\gamma/2}\lambda_{j+1}^{t+1}\mu_{j+1}^{s}
\Big(\frac{(B_j/\tilde C_j)^{1/2}}{\lambda_{j+1}/\lambda_j} +1\Big),
\\ & \|\partial_1^{(t)}\partial_2^{(s)} \partial_{22}v^2_{j+1}\|_0\leq
C \tilde C_j^{1/2} \mu_{j}^{\gamma/2}\lambda_{j+1}^{t}\mu_{j+1}^{s+1}
\Big(\frac{(B_j/\tilde C_j)^{1/2}}{\mu_{j+1}/\mu_j} +1\Big).
\end{split}
\end{equation*}
This yields \ref{P2bound2} recalling the assumption (\ref{ass_ce2}).

\smallskip

{\bf 8. (Proof of \ref{P2bound3} - corrugations bounds) } Recall the decomposition:
\begin{equation}\label{if}
w_{j+1} - w_j = (W_1-w_j) + (W_2-W_1) + (W_3-W_2) + (w_{j+1}-W_3).
\end{equation}
By (\ref{VW1}), \ref{Ebound23}, \ref{Ebound43}, \ref{ass_HQ3} and
the assumption (\ref{ass_ce2}) we get:
\begin{equation*}
\begin{split}
& \|W_1 - w_j\|_1\leq C\Big(\|\Psi_{j+1}\|_1
 + \|a_{j+1}\|_0\|\nabla v_j^1\|_0 + \|a_{j+1}\|_0^2 
\\ & \qquad\qquad \qquad \qquad  +\frac{\|\nabla a_{j+1}\|_0\|\nabla
  v_j^1\|_0 + \|a_{j+1}\|_0\|\nabla^2 v_{j}^1\|_0 +
  \|a_{j+1}\|_0 \|\nabla a_{j+1}\|_0}{\lambda_{j+1}} \\
& \qquad\qquad \qquad \qquad  +\frac{\| a_{j+1}\|_0\|\nabla^2
  a_{j+1}\|_0 + \|\nabla a_{j+1}\|_0^2}{\lambda_{j+1}^2} \Big)\\
& \leq C \tilde C_j^{1/2}\mu_j^{\gamma/2} 
\Big(\|\nabla v_j^1\|_0 + \tilde C_j^{1/2}\mu_j^{\gamma/2}  
+\frac{B_j^{1/2}\mu_{j-1}^{\gamma/2}}{\lambda_{j+1}/\lambda_j}\Big)
\leq C \tilde C_j^{1/2}\mu_j^{\gamma/2} \big(\|\nabla v_j^1\|_0 +
\tilde C_j^{1/2}\mu_j^{\gamma/2} \big),
\end{split}
\end{equation*}
with constants $C$ depending on $\omega,\gamma, N$. Similarly:
\begin{equation*}
\begin{split}
& \|\nabla^2 (W_1 - w_{j})\|_0\leq 
C \Big( \|\nabla^2\Psi_{j+1}\|_0 + 
\lambda_{j+1}\big(\|a_{j+1}\|_0\|\nabla v_j^1\|_0+ \|a_{j+1}\|_0^2\big) 
\\ & \qquad\qquad\qquad \qquad \qquad
+ \big(\|\nabla a_{j+1}\|_0\|\nabla v_j^1\|_0 + \|a_{j+1}\|_0\|\nabla^2 v_j^1\|_0  
+ \|a_{j+1}\|_0\|\nabla a_{j+1}\|_0\big)
\\ & \qquad\qquad\qquad \qquad\qquad + \frac{\|\nabla^2a_{j+1}\|_0\|\nabla v_j^1\|_0 + \|\nabla
  a_{j+1}\|_0\|\nabla^2 v_j^1\|_0 + \|a_{j+1}\|_0\|\nabla^3 v_j^1\|_0}{\lambda_{j+1}} 
\\ & \qquad\qquad\qquad\qquad\qquad + \frac{\|a_{j+1}\|_0\|\nabla^2a_{j+1}\|_0  + \|\nabla
  a_{j+1}\|_0^2}{\lambda_{j+1}}
\\ & \qquad\qquad\qquad \qquad\qquad 
+ \frac{\|a_{j+1}\|_0 \|\nabla^3a_{j+1}\|_0 + \|\nabla
  a_{j+1}\|_0\|\nabla^2 a_{j+1}\|_0}{\lambda_{j+1}^2}\Big) 
\\ & 
 \qquad\qquad \qquad \quad \leq C \tilde C_j^{1/2}\mu_j^{\gamma/2} 
\big(\|\nabla v_j^1\|_0 + \tilde C_j^{1/2}\mu_j^{\gamma/2} \big)\lambda_{j+1}.
\end{split}
\end{equation*}
For the bounds of the difference of $W_3$ and $W_2$, we
write: 
\begin{equation*}
\begin{split}
& \|W_3 - W_2\|_1\leq C\Big(\|b_{j+1}\|_0\|\nabla v_j^2\|_0 + \|b_{j+1}\|_0^2 
\\ & \qquad\qquad \qquad \qquad  +\frac{\|\nabla b_{j+1}\|_0\|\nabla
  v_j^2\|_0 + \|b_{j+1}\|_0\|\nabla^2 v_{j}^2\|_0 +
  \|b_{j+1}\|_0 \|\nabla b_{j+1}\|_0}{\mu_{j+1}} \\
& \qquad\qquad \qquad \qquad  +\frac{\|b_{j+1}\|_0\|\nabla^2
  b_{j+1}\|_0 + \|\nabla b_{j+1}\|_0^2}{\mu_{j+1}^2} \Big),\\
& \|\nabla^2 (W_3 - W_2)\|_0\leq 
C \Big(\mu_{j+1}\big(\|b_{j+1}\|_0\|\nabla v_j^2\|_0+ \|b_{j+1}\|_0^2\big) 
\\ & \qquad\qquad\qquad \qquad \qquad
+ \big(\|\nabla b_{j+1}\|_0\|\nabla v_j^2\|_0 + \|b_{j+1}\|_0\|\nabla^2 v_j^2\|_0  
+ \|b_{j+1}\|_0\|\nabla b_{j+1}\|_0\big)
\\ & \qquad\qquad\qquad \qquad\qquad + \frac{\|\nabla^2b_{j+1}\|_0\|\nabla v_j^2\|_0 + \|\nabla
  b_{j+1}\|_0\|\nabla^2 v_j^2\|_0 + \|b_{j+1}\|_0\|\nabla^3 v_j^2\|_0}{\mu_{j+1}} 
\\ & \qquad\qquad\qquad\qquad\qquad + \frac{\|b_{j+1}\|_0\|\nabla^2b_{j+1}\|_0  + \|\nabla
  b_{j+1}\|_0^2}{\mu_{j+1}}
\\ & \qquad\qquad\qquad \qquad\qquad 
+ \frac{\|b_{j+1}\|_0 \|\nabla^3b_{j+1}\|_0 + \|\nabla
  b_{j+1}\|_0\|\nabla^2 b_{j+1}\|_0}{\mu_{j+1}^2}\Big), 
\end{split}
\end{equation*}
and use use (\ref{VW3}), \ref{b1}, \ref{b2} and \ref{ass_HQ3}, (\ref{ass_ce2}) to obtain:
\begin{equation*}
\begin{split}
& \|W_3 - W_2\|_1\leq C \tilde C_j^{1/2}\mu_j^{\gamma/2} 
\Big(\|\nabla v_j^2\|_0 + \tilde C_j^{1/2}\mu_j^{\gamma/2}  
+\frac{B_j^{1/2}\mu_{j-1}^{\gamma/2}}{\mu_{j+1}/\mu_{j}}\Big)
\\ & \qquad \qquad\quad \; \leq C \tilde C_j^{1/2}\mu_j^{\gamma/2} \big(\|\nabla v_j^1\|_0 +
\tilde C_j^{1/2}\mu_j^{\gamma/2} \big),\vspace{1mm}\\
& \|\nabla^2 (W_3 - W_2)\|_0\leq C \tilde C_j^{1/2}\mu_j^{\gamma/2} 
\big(\|\nabla v_j^2\|_0 + \tilde C_j^{1/2}\mu_j^{\gamma/2} \big)\mu_{j+1}.
\end{split}
\end{equation*}

\smallskip

{\bf 9. (Proof of \ref{P2bound3} - decomposition bounds) } 
Pertaining to the application of Lemma \ref{lem_IBP} and Corollary \ref{cor_IBP},
we estimate the difference of $W_2$ and $W_1$ using (\ref{VW2}),
\ref{Ebound23}, \ref{ass_HQ3}, (\ref{ass_ce2}):
\begin{equation*}
\begin{split}
& \|W_2 - W_1\|_1 \leq C\sum_{i=0}^{N+1}
\frac{\|\nabla^{(i)}\big(a_{j+1}\nabla^2v_j^1\big)\|_0}{\lambda_{j+1}^{i+1}} 
+ C \sum_{i=0}^{N+1}
\frac{\|\nabla^{(i)}\big(a_{j+1}\nabla^2a_{j+1}\big)\|_0 +
\|\nabla^{(i)}\big(\nabla a_{j+1}^{\otimes 2} \big)\|_0  }{\lambda_{j+1}^{i+2}} 
\\ & \qquad\qquad\quad \; 
\leq C \tilde C_j^{1/2}\mu_j^{\gamma/2}\sum_{i=1}^{N+1}\Big(
\frac{B_j^{1/2}\mu_{j-1}^{\gamma/2}}{(\lambda_{j+1}/\lambda_j)^{i+1}} + 
\frac{\tilde C_j^{1/2}\mu_{j}^{\gamma/2}}{(\lambda_{j+1}/\mu_j)^{i+2}}\Big)
\leq C\tilde C_j \mu_j^\gamma.
\end{split}
\end{equation*}
Similarly, the second derivatives' difference satisfies:
\begin{equation*}
\begin{split}
& \|\nabla^2(W_2 - W_1)\|_0 \leq C\sum_{i=0}^{N+2}
\Big(\frac{\|\nabla^{(i)}\big(a_{j+1}\nabla^2v_j^1\big)\|_0}{\lambda_{j+1}^{i}} 
+ \frac{\|\nabla^{(i)}\big(a_{j+1}\nabla^2a_{j+1}\big)\|_0 +
\|\nabla^{(i)}\big(\nabla a_{j+1}^{\otimes 2} \big)\|_0  }{\lambda_{j+1}^{i+1}} \Big)
\\ & \qquad\qquad\qquad \quad \; \, \leq C\tilde C_j \mu_j^\gamma\lambda_{j+1},
\end{split}
\end{equation*}
whereas to estimate the difference of $w_{j+1}$ and $W_3$ we
additionally use (\ref{VW4}), \ref{b1}, \ref{b2}:
\begin{equation*}
\begin{split}
& \|w_{j+1} - W_3\|_1 \leq C\sum_{i=0}^{2}
\frac{\|\nabla^{(i)}\big(b_{j+1}\nabla^2v_j^2\big)\|_0}{\mu_{j+1}^{i+1}} 
\leq C \tilde C_j^{1/2}\mu_j^{\gamma/2}\sum_{i=0}^{2}
\frac{B_j^{1/2}\mu_{j-1}^{\gamma/2}}{(\mu_{j+1}/\lambda_{j+1})^{i+1}} 
\leq C\tilde C_j \mu_j^\gamma, \\
& \|\nabla^2(w_{j+1} - W_3)\|_0 \leq C\sum_{i=0}^{3}
\frac{\|\nabla^{(i)}\big(b_{j+1}\nabla^2v_j^2\big)\|_0}{\mu_{j+1}^{i}} 
\leq C\tilde C_j \mu_j^\gamma\mu_{j+1},
\end{split}
\end{equation*}
with constants $C$ depending on $\omega,\gamma, N$. 
Combining the bounds of steps 8 and 9 for the four quantities
appearing in (\ref{if}), we arrive at \ref{P2bound3}. The proof of
Proposition \ref{prop2} is done.
\end{proof}

\section{The stage construction and a proof of Theorem \ref{thm_stage}} \label{sec_khamsa} 

We are now ready to complete the ``stage'' in the present version of the 
convex integration algorithm. It requires a given number $K$ of iterations of the
quadruple step construction put forward in Proposition
\ref{prop2}. Since the first two iterations do not satisfy assumptions (\ref{ass_ce2}) and \ref{ass_HQ3}, they have to be
considered separately, with the same components of the proof of
Proposition \ref{prop2} carried out in steps 2-5 of the proof below. As we argue in the proof, the natural choice for the progression of frequencies is as depicted in Figure \ref{fig_freq}.

\begin{figure}[htbp]
\centering
\includegraphics[scale=0.6]{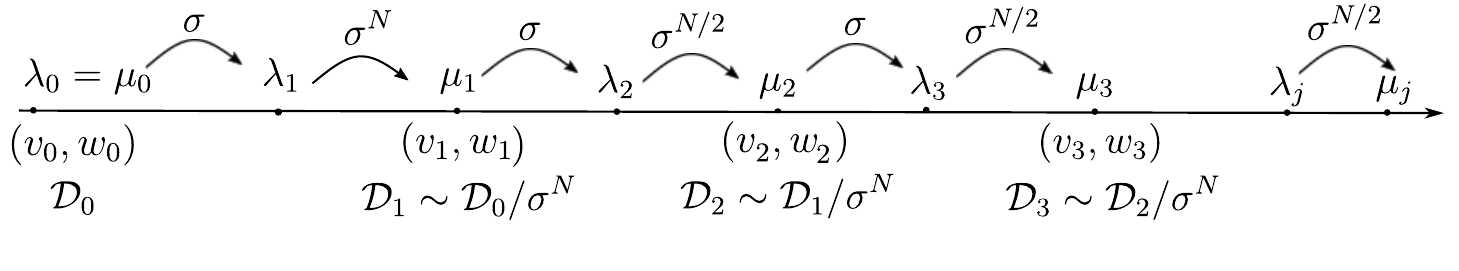}
\caption{{Progression of frequencies and consecutive defects' magnitudes in the proof of Theorem \ref{thm_stage}.}}
\label{fig_freq}
\end{figure}

\bigskip

\noindent {\bf Proof of Theorem \ref{thm_stage}}.

{\bf 1. (Setting the initial quantities)}
For the given double iteration numerals $N, K\geq 1$ and the
regularity exponent $\gamma\in (0,1)$, 
we take $l_0$ and $\sigma_0$ as in Proposition \ref{prop2}. Let $v, w$ be as
in the statement of the theorem, together with the positive constants
$l,\lambda,\mathcal{M}$, where we denote:
$$\sigma = \lambda l$$
and assume the following version of (\ref{Assu}):
\begin{equation}\label{Addi}
l\leq l_0,\qquad \sigma l^\gamma\geq \sigma^{K(1+N/2)\gamma}\sigma_0, \qquad
\mathcal{M}\geq \max\{\|v\|_0,\|w\|_0, 1\}.
\end{equation}
Resolving (\ref{Addi}) against (\ref{Assu}) will
be the content of step 8 below. We now observe that the second assumption
in (\ref{Addi}) implies that $\lambda_{j+1}^{1-\gamma} \geq
\mu_j\sigma_0$ for all $j=0\ldots K-1$, upon defining:
\begin{equation}\label{DEF}
\begin{split}
& \mu_0 = \frac{1}{l}, \qquad\mu_j = \frac{\sigma^{j+(j+1)N/2}}{l}\quad
\mbox{ for all }\; j=1\ldots K,\\
& \qquad\qquad\quad \; \;  
\lambda_j = \frac{\sigma^{j(1+N/2)}}{l}\qquad \mbox{ for all }\; 0=1\ldots K,
\end{split}
\end{equation}
because then indeed:
$$\frac{\lambda_{j+1}^{1-\gamma}}{\mu_j} =
\sigma^{1-(j+1)(1+N/2)\gamma}l^\gamma \geq
\sigma^{1-K(1+N/2)\gamma}l^\gamma\geq\sigma_0 \quad \mbox{ for all
}\; j=0\ldots K.$$
We first construct the fields 
$v_0, w_0\in \mathcal{C}^\infty(\bar\omega+\bar B_{l}(0),\R^2)$, 
$A_0\in \mathcal{C}^\infty(\bar\omega+\bar B_{l}(0),\R^{2\times 2}_\sym)$ 
by using the mollification kernel as in Lemma \ref{lem_stima}:
$$v_0=v\ast \phi_{l},\quad w_0=w\ast \phi_{l}, \quad A_0=A\ast \phi_{l},
\quad {\mathcal{D}}_0= A_0 - \big(\frac{1}{2}(\nabla v_0)^T\nabla v_0 + \sym\nabla w_0\big).$$
From Lemma \ref{lem_stima}, we deduce the initial bounds,
where constants $C$ depend only on $\omega, m, N, K$:
\begin{align*}
& \|v_0-v\|_1 + \|w_0-w\|_1 \leq C l\mathcal{M},
\tag*{(\theequation)$_1$}\refstepcounter{equation} \label{pr_stima1}\\
& \|A_0-A\|_0 \leq l^{s+\beta}\|A\|_{s,\beta}, \tag*{(\theequation)$_2$} \label{pr_stima2}\\
& \|\nabla^{(m)}\nabla^2v_0\|_0 + \|\nabla^{(m)}\nabla^2w_0\|_0\leq
\frac{C}{l^m} \mathcal{M}\quad \mbox{ for all }\; m=0\ldots K(3N+4), \tag*{(\theequation)$_3$} \label{pr_stima3}\\
& \|\nabla^{(m)} \mathcal{D}_0\|_0\leq
\frac{C}{l^m} \big(\|\mathcal{D}\|_0 + (l\mathcal{M})^2\big) \qquad \quad \mbox{ for
  all }\; m=0\ldots K(3N+4). \tag*{(\theequation)$_4$}\label{pr_stima4} 
\end{align*}
Indeed,  \ref{pr_stima1}, \ref{pr_stima2} follow from \ref{stima2} and
in view of the lower bound on $\mathcal{M}$. Similarly,
\ref{pr_stima3} follows by applying \ref{stima1} to 
$\nabla^2v$ and $\nabla^2w$ with the differentiability exponent $m-1$.
Since:
$$\mathcal{D}_0 = \mathcal{D}\ast \phi_{l} - \frac{1}{2}\big((\nabla
v_0)^T\nabla v_0 - ((\nabla v)^T\nabla v)\ast\phi_{l}\big), $$ 
we get \ref{pr_stima4} by applying \ref{stima1} to $\mathcal{D}$, and
\ref{stima4} to $\nabla v$.

\smallskip

{\bf 2. (The first iteration: from $j=0$ to $1$; the defect bound)} 
We follow the proof of Proposition \ref{prop2}
starting with $v_0, w_0, \mathcal{D}_0$ and parameters consistent with
\ref{pr_stima4}, \ref{pr_stima3}, (\ref{DEF}):
$$\tilde C_0=C \big(\|\mathcal{D}\|_0 + (l\mathcal{M})^2\big), \quad
B_0= \tilde C_0,\quad \lambda_0=\mu_0=\frac{1}{l}, \quad \lambda=\lambda_1<\mu_1.$$
Recall that $\lambda_1^{1-\gamma}\geq \mu_0\sigma_0$ by the second
assumption in (\ref{Addi}). Thus Proposition \ref{prop1} yields the fields
$a_{1}\in \mathcal{C}^\infty(\bar\omega +
\bar B_l(0), \R)$, $\Psi_{1}\in\mathcal{C}^\infty(\bar\omega +
\bar B_l(0),\R^2)$ and $\mathcal{F}_{1}\in
\mathcal{C}^\infty(\bar\omega+\bar B_l(0),\R^{2\times 2}_\sym)$
which satisfy same bounds as in  \ref{Ebound13}-\ref{Ebound43}:
\begin{equation}\label{a1}
\begin{split}
&  \frac{\bar C^{1/2} \tilde C_0^{1/2}}{2}\mu_0^{\gamma/2}\leq
a_{1}\leq \frac{3\bar C^{1/2} \tilde C_0^{1/2}}{2}\mu_0^{\gamma/2}
\; \quad \mbox{ and }\quad \|\nabla^{(m)} a_{1}\|_0\leq C\tilde
C_0^{1/2}\mu_0^{\gamma/2}\mu_0^m, \vspace{1mm} \\ 
& \displaystyle{\|\nabla^{(m)}\mathcal{F}_{1}\|_0 \leq 
C \tilde C_0 \frac{\mu_{0}^m}{(\lambda_{1}/\mu_0)^N}},\\
& \|\Psi_{1}\|_1\leq C\tilde C_0\mu_0^\gamma, \qquad
\|\nabla^2\Psi_{1}\|_0\leq C\tilde C_0\mu_0^\gamma\mu_0.
\end{split}
\end{equation}
Bounds (\ref{bound_Gj+1}) and (\ref{bound_G}) likewise remain the same:
\begin{equation}\label{G1}
\|\nabla^{(m)}\mathcal{G}_{1}\|_0\leq C \tilde C_0
\frac{\lambda_{1}^m}{(\lambda_{1}/\mu_0)^{N}},
\qquad \|\nabla^{(m)}G_1\|_0\leq C \tilde C_0
\frac{\mu_0^{\gamma}}{\lambda_{1}/\mu_0} \lambda_{1}^m.
\end{equation}
which together with (\ref{a1}) implies that $b_1=(a_1^2+G_1)^{1/2} \in \mathcal{C}^\infty(\bar\omega +
\bar B_l(0), \R)$ is well defined and satisfies, as in \ref{b1} and \ref{b2}:
\begin{equation}\label{B1}
\|b_{1}\|_0\leq C\tilde C_0^{1/2}\mu_0^{\gamma/2}
\; \quad \mbox{ and }\quad \|\nabla^{(m)} b_{1}\|_0\leq C\tilde
C_0^{1/2}\mu_0^{\gamma/2}\frac{\lambda_1^m}{\lambda_1/\mu_0} \mbox{ for } \; m\geq 1. 
\end{equation}
Consequently, the bound as in (\ref{bound_Hj+1}) is still the same, namely:
\begin{equation}\label{H1}
\|\nabla^{(m)}\mathcal{H}_1\|_0\leq C\tilde C_0^{1/2}\frac{\mu_1^m}{(\mu_1/\lambda_1)^2}.
\end{equation}
Now the fourth error term $\mathcal{I}_1$, defined in (\ref{err4}), enjoys the
worse bound than that in (\ref{bound_Ij+1}):
\begin{equation}\label{I1}
\begin{split}
\|\partial_1^{(t)}\partial_2^{(s)} \mathcal{I}_1\|_0 & \leq C\sum_{i=0}^2\sum_{p+q=s}
\mu_1^{p-i-1}\|\nabla^{(i+q+t)}\big(b_1\nabla^2v_0^2\big)\|_0 \\ & \leq
C\tilde C_0\mu_0^{\gamma/2}\sum_{i=0}^2\sum_{p+q=s}
\mu_1^{p-i-1}\lambda_1^{i+q+t}\mu_0 
\leq C\tilde C_0\mu_0^{\gamma/2}\frac{\lambda_1^t\mu_1^s}{\mu_1/\mu_0}
\leq C\tilde C_0\frac{\lambda_1^t\mu_1^s}{\mu_1/\lambda_1}.
\end{split}
\end{equation}
because although \ref{pr_stima3} is as before, derivatives of the different
components of $\nabla^2v_0$ have the same bound in (\ref{B1}).
Combining (\ref{a1}), (\ref{G1}), (\ref{H1}), (\ref{I1}) we obtain the principal bound:
\begin{equation}\label{D1}
\begin{split}
\|\nabla^{(m)} \mathcal{D}_1\|\leq C \tilde C_0\mu_1^m\Big(\frac{1}{(\lambda_1/\mu_0)^N}
& + \frac{1}{\mu_1/\lambda_1}\Big)
\leq C \tilde C_0 \frac{\mu_1^m}{\mu_1/\lambda_1}
\\ & \mbox{for all }\; m=0\ldots (K-1)(3N+4),
\end{split}
\end{equation}
which corresponds to \ref{P2bound4}, by setting according to (\ref{DEF}):
\begin{equation}\label{one}
\mu_1/\lambda_1 = (\lambda_1/\mu_0)^N = \sigma^N.
\end{equation}

\smallskip

{\bf 3. (The first iteration: from $j=0$ to $1$; the remaining bounds)} 
Recalling (\ref{VW1}), (\ref{VW3}) and using (\ref{a1}), (\ref{B1}) there follows:
\begin{equation}\label{fir_v1}
\|v_1-v_0\|_1\leq C\tilde C_0^{1/2}\mu_0^{\gamma/2},
\end{equation}
together with the bounds on $\nabla^2v_1$, corresponding yet inferior to \ref{P2bound2}, valid with uniform constants $C$ depending only on $\omega, \gamma, N, K$ for all $t+s=0\ldots (K-1)(3N+4)$:
\begin{equation}\label{sec_v1}
\begin{split}
& \|\partial_1^{(t)}\partial_2^{(s)}\partial_{11}v_1^1\|_0\leq \tilde C_0^{1/2}\mu_0^{t+s+1}
+ C\tilde C_0^{1/2}\mu_0^{\gamma/2}\lambda_1^{t+1}\mu_0^s\leq C\tilde
C_0^{1/2}\mu_0^{\gamma/2}\lambda_1^{t+1}\mu_0^s, \\  
& \|\partial_1^{(t)}\partial_2^{(s)}\partial_{12}v_1^1\|_0 
\leq \tilde C_0^{1/2}\mu_0^{t+s+1}
+ C\tilde C_0^{1/2}\mu_0^{\gamma/2}\lambda_1^{t}\mu_0^{s+1}
\leq C\tilde C_0^{1/2}\mu_0^{\gamma/2}\lambda_1^{t}\mu_0^{s+1}, \\  
& \|\partial_1^{(t)}\partial_2^{(s)}\partial_{22}v_1^1\|_0
\leq \tilde C_0^{1/2}\mu_0^{t+s+1} 
+ C\tilde C_0^{1/2}\mu_0^{\gamma/2}\lambda_1^{t-1}\mu_0^{s+2}
\leq C\tilde C_0^{1/2}\mu_0^{\gamma/2}\lambda_1^{t}\mu_0^{s+1}, \vspace{2mm}\\
& \|\partial_1^{(t)}\partial_2^{(s)}\partial_{11}v_1^2\|_0 
\leq \tilde C_0^{1/2}\mu_0^{t+s+1}
+ C\tilde C_0^{1/2}\mu_0^{\gamma/2}\lambda_1^{t+2}\mu_1^{s-1}
\leq C\tilde C_0^{1/2}\mu_0^{\gamma/2}\lambda_1^{t+1}\mu_1^{s}, \\  
& \|\partial_1^{(t)}\partial_2^{(s)}\partial_{12}v_1^2\|_0
\leq \tilde C_0^{1/2}\mu_0^{t+s+1} 
+ C\tilde C_0^{1/2}\mu_0^{\gamma/2}\lambda_1^{t+1}\mu_1^{s}
\leq C\tilde C_0^{1/2}\mu_0^{\gamma/2}\lambda_1^{t+1}\mu_1^{s},\\ 
& \|\partial_1^{(t)}\partial_2^{(s)}\partial_{22}v_1^2\|_0
\leq \tilde C_0^{1/2}\mu_0^{t+s+1} 
+ C\tilde C_0^{1/2}\mu_0^{\gamma/2}\lambda_1^{t}\mu_1^{s+1}
\leq C\tilde C_0^{1/2}\mu_0^{\gamma/2}\lambda_1^{t}\mu_1^{s+1}.
\end{split}
\end{equation}
Carrying out the remaining bounds as in steps 8 and 9 of the proof of
Proposition \ref{prop2}, we get:
\begin{equation}\label{sec_w1}
\begin{split}
& \|w_1-w_0\|_1\leq C\tilde C_0^{1/2}\mu_0^{\gamma/2}\big(\|\nabla v_0\|_0
+ \tilde C_0^{1/2}\mu_0^{\gamma/2}\big)\\ & \qquad\qquad\quad \leq
C\tilde C_0^{1/2}\mu_0^{\gamma/2}\big(\|\nabla v\|_0 + l\mathcal{M}
+ \tilde C_0^{1/2}\mu_0^{\gamma/2}\big) \leq
C\tilde C_0^{1/2}\mu_0^{\gamma}\big(\|\nabla v\|_0 + \tilde C_0^{1/2}\big), 
\\ & \|\nabla^{2}(w_1-w_0)\|_0\leq 
C\tilde C_0^{1/2}\mu_0^{\gamma}\big(\|\nabla v\|_0 + \tilde C_0^{1/2}\big)\mu_1,
\end{split}
\end{equation}
in view of \ref{pr_stima1}, \ref{pr_stima3}, (\ref{a1}) and (\ref{B1}).

\smallskip

{\bf 4. (The second iteration: from $j=1$ to $2$; the defect bound)} 
We follow  the proof of Proposition \ref{prop2}
starting with $v_1, w_1, \mathcal{D}_1$ and parameters consistent with
(\ref{D1}), (\ref{sec_v1}) and (\ref{DEF}):
$$\tilde C_1=C \frac{\tilde C_0}{\mu_1/\lambda_1}, \quad
B_1= \tilde C_0,\quad \mu_1\frac{\lambda_1}{\mu_0}=\mu_1\sigma =\lambda_2<\mu_2,$$
so that $\|\nabla^{(m)}\mathcal{D}_1\|_0\leq \tilde C_1\mu_1^m$ and
that (\ref{sec_v1}) is valid with $C\tilde C_0^{1/2}$ in its right hand side
replaced by $\tilde B_1^{1/2}$. Recall from (\ref{Addi}) that
$\lambda_2^{1-\gamma}\geq \mu_1\sigma_0$, allowing for the application
of Proposition \ref{prop1} which yields $a_{2}\in \mathcal{C}^\infty(\bar\omega +
\bar B_l(0), \R)$, $\Psi_{2}\in\mathcal{C}^\infty(\bar\omega +
\bar B_l(0),\R^2)$ and $\mathcal{F}_{2}\in
\mathcal{C}^\infty(\bar\omega+\bar B_l(0),\R^{2\times 2}_\sym)$.
These satisfy, as in \ref{Ebound13}-\ref{Ebound43}:
\begin{equation}\label{a2}
\begin{split}
&  \frac{\bar C^{1/2} \tilde C_1^{1/2}}{2}\mu_1^{\gamma/2}\leq
a_{2}\leq \frac{3\bar C^{1/2} \tilde C_1^{1/2}}{2}\mu_1^{\gamma/2}
\; \quad \mbox{ and }\quad \|\nabla^{(m)} a_{2}\|_0\leq C\tilde
C_1^{1/2}\mu_1^{\gamma/2}\mu_1^m, \vspace{1mm} \\ 
& \displaystyle{\|\nabla^{(m)}\mathcal{F}_{2}\|_0 \leq 
C \tilde C_1 \frac{\mu_{1}^m}{(\lambda_{2}/\mu_1)^N}},\\
& \|\Psi_{1}\|_1\leq C\tilde C_1\mu_1^\gamma, \qquad
\|\nabla^2\Psi_{2}\|_0\leq C\tilde C_1\mu_1^\gamma\mu_1.
\end{split}
\end{equation}
The bound (\ref{bound_Gj+1}) likewise remains the same, as 
$\big({\tilde C_0}/{\tilde C_1}\big)^{1/2}\leq
\big({\mu_1}/{\lambda_1}\big)^{1/2}\leq{\lambda_2}/{\lambda_1}$, and then:
\begin{equation}\label{G2}
\|\nabla^{(m)}\mathcal{G}_{2}\|_0\leq C \tilde C_1
\frac{\lambda_{2}^m}{(\lambda_{2}/\mu_1)^{N}}.
\end{equation}
Regarding (\ref{bound_G}), we only need to estimate derivatives of $IV_0=
\frac{\Gamma(\lambda_2x_1)}{\lambda_2}a_2\partial_{22}v_1^1$ as in step 4 of
the proof of Proposition \ref{prop2}, using (\ref{sec_v1}) and (\ref{a2}):
\begin{equation*}
\begin{split}
& \|\partial_1^{(t)}\partial_2^{(s)}IV_0\|_0  
\leq  C\sum_{p+q=t} \lambda_{2}^{p-1}\|\partial_1^{(q)}\partial_2^{(s)}
\big(a_{2}\partial_{22}v_1^1\big)\|_0 \\ & 
\qquad\qquad\qquad \leq C \hspace{-0.5cm}
\sum_{\tiny\begin{array}{c} p+q=t \\  u_1+z_1 = q \\ u_2+z_2=s \end{array}} \hspace{-0.5cm}
\lambda_{2}^{p-1}\|\partial_1^{(u_1)}\partial_2^{(u_2)} a_{2}\|_0
\|\partial_1^{(z_1)}\partial_2^{(z_2+2)} v_1^1\|_0 
\\ & \leq C \tilde C_1^{1/2}\tilde C_0^{1/2}\mu_1^{\gamma} \hspace{-0.5cm}
\sum_{\tiny\begin{array}{c} p+q=t \\  u_1+z_1 = q \\ u_2+z_2=s \end{array}} \hspace{-0.5cm}
\lambda_{2}^{p-1}\mu_1^{u_1+u_2}\lambda_1^{z_1}\mu_{0}^{z_2+1}
\leq  C \tilde C_1^{1/2}\tilde C_0^{1/2}\mu_1^{\gamma} \sum_{p+q=t}
\lambda_{2}^{p-1} \mu_1^{q+s}\mu_{0} \\ & 
\leq C \tilde C_1^{1/2}\tilde C_0^{1/2}\mu_1^{\gamma} 
\frac{\lambda_{2}^t\mu_1^s}{\lambda_{2}/\mu_{0}} 
= C \tilde C_1\lambda_{2}^t\mu_1^s\mu_1^{\gamma} 
\frac{(\tilde C_0/\tilde C_1)^{1/2}}{\lambda_{2}/\mu_{0}},
\end{split}
\end{equation*}
resulting in, in view of (\ref{Assu}) and (\ref{D1}):
\begin{equation*}%\label{G2}
\begin{split}
\|\nabla^{(m)}G_2\|_0& \leq C \tilde C_1\lambda_2^m\frac{\mu_1^\gamma}{\lambda_2/\mu_1}
\cdot \frac{(\tilde C_0/\tilde
  C_1)^{1/2}}{(\mu_1/\lambda_1)\min\big\{\lambda_1/\mu_0, \;
  (\lambda_2/\mu_1)(\lambda_1/\mu_0), \; (\lambda_2/\mu_1)^2\big\}}
\\ & \quad +  C \tilde C_1\frac{\lambda_2^m}{\lambda_2/\mu_1}
\leq C\tilde C_1\frac{\lambda_2^m}{\lambda_2/\mu_1}.
\end{split}
\end{equation*}
Together with (\ref{a2}) the above implies that $b_2=(a_2^2+G_2)^{1/2}
\in \mathcal{C}^\infty(\bar\omega + 
\bar B_l(0), \R)$ is well defined and satisfies, as in \ref{b1} and \ref{b2}:
\begin{equation}\label{B2}
\|b_{2}\|_0\leq C\tilde C_1^{1/2}\mu_1^{\gamma/2}
\; \quad \mbox{ and }\quad \|\nabla^{(m)} b_{2}\|_0\leq C\tilde
C_1^{1/2}\mu_1^{\gamma/2}\frac{\lambda_2^m}{\lambda_2/\mu_1} \mbox{ for } \; m\geq 1. 
\end{equation}
Consequently, the bound as in (\ref{bound_Hj+1}) is still the same, namely:
\begin{equation}\label{H2}
\|\nabla^{(m)}\mathcal{H}_2\|_0\leq C\tilde C_1^{1/2}\frac{\mu_2^m}{(\mu_2/\lambda_2)^2}.
\end{equation}
Towards estimating the fourth error term $\mathcal{I}_2$ in (\ref{err4}),
we proceed as in step 6 of the proof of Proposition \ref{prop2},
noting the new estimate in view of the inferior bound on $\partial_{11}v_1^2$ in (\ref{sec_v1}):
\begin{equation*}
\begin{split}
& \sum_{p+q=s} \mu_{2}^{p-1}\|\partial_1^{(t)}\partial_2^{(q)}
\big(b_{2}\partial_{11}v_1^2\big)\|_0 
\leq C \hspace{-0.5cm}
\sum_{\tiny\begin{array}{c} p+q=s \\  u_1+z_1 = t \\ u_2+z_2=q \end{array}} \hspace{-0.5cm}
\mu_{2}^{p-1}\|\partial_1^{(u_1)}\partial_2^{(u_2)} b_{2}\|_0
\|\partial_1^{(z_1+2)}\partial_2^{(z_2)} v_1^2\|_0 
\\ & \leq C \tilde C_1^{1/2}\tilde C_0^{1/2}\mu_1^{\gamma} \hspace{-0.5cm}
\sum_{\tiny\begin{array}{c} p+q=s \\  u_1+z_1 = t \\ u_2+z_2=q \end{array}} \hspace{-0.5cm}
\mu_{2}^{p-1} \lambda_2^{u_1+u_2}\lambda_1^{z_1+1}\mu_{1}^{z_2}
\leq  C \tilde C_1^{1/2}\tilde C_0^{1/2}\mu_1^{\gamma} \sum_{p+q=s}
\mu_{2}^{p-1} \lambda_2^{q+t}\lambda_1 \\ & 
\leq C \tilde C_1^{1/2}\tilde C_0^{1/2}\mu_1^{\gamma} 
\frac{\lambda_{2}^t\mu_2^s}{\mu_{2}/\lambda_{1}} 
\leq C \tilde C_1\lambda_{2}^t\mu_2^s
\frac{(\tilde C_0/\tilde C_1)^{1/2}}{(\mu_{2}/\lambda_1)(\mu_1/\lambda_{1})},
\end{split}
\end{equation*}
and concluding the bound corresponding to (\ref{bound_Ij+1}):
\begin{equation}\label{I2}
\begin{split}
\|\nabla^{(m)}\mathcal{I}_2\|_0& \leq C \tilde C_1\mu_2^m
\frac{(\tilde C_0/\tilde C_1)^{1/2}}
{(\mu_2/\lambda_2)\min\big\{(\mu_2/\lambda_2)^2, \;
  \mu_1/\lambda_1, \; (\mu_2/\lambda_2)(\mu_1/\lambda_1)\big\}}.
\end{split}
\end{equation}
Together with (\ref{a2}), (\ref{G2}), (\ref{H2}), the above implies
the second principal bound:
\begin{equation}\label{D2}
\begin{split}
\|\nabla^{(m)} \mathcal{D}_2\|\leq C \tilde C_1\mu_2^m\Big(\frac{1}{(\lambda_2/\mu_1)^N}
+ & \frac{1}{(\mu_2/\lambda_2)^2}\Big)
\leq C \tilde C_1\frac{\mu_1^m}{(\mu_2/\lambda_2)^2}
\\ & \mbox{for all }\; m=0\ldots (K-2)(3N+4),
\end{split}
\end{equation}
which corresponds to \ref{P2bound4}, by setting in accordance with (\ref{DEF}):
\begin{equation}\label{two}
\mu_2/\lambda_2 = (\mu_1/\lambda_1)^{1/2}= (\lambda_2/\mu_1)^{N/2} = \sigma^{N/2}.
\end{equation}

\smallskip

{\bf 5. (The second iteration: from $j=1$ to $2$;  the remaining bounds)} 
Recalling (\ref{VW1}), (\ref{VW3}) and using (\ref{a2}), (\ref{B2}), there follows:
\begin{equation}\label{fir_v2}
\|v_2-v_1\|_1\leq C\tilde C_1^{1/2}\mu_1^{\gamma/2},
\end{equation}
together with the bounds on $\nabla^2v_2$, %corresponding to \ref{P2bound2}, 
valid for all $t+s=0\ldots (K-2)(3N+4)$:
\begin{equation}\label{sec_v2}
\begin{split}
& \|\partial_1^{(t)}\partial_2^{(s)}\partial_{11}v_2^1\|_0\leq C\tilde
C_0^{1/2}\mu_0^{\gamma/2}\lambda_1^{t+1}\mu_0^s
+ C\tilde C_1^{1/2}\mu_1^{\gamma/2}\lambda_2^{t+1}\mu_1^s
\leq C\tilde C_1^{1/2}\mu_1^{\gamma/2}\lambda_2^{t+1}\mu_1^s,\\  
& \|\partial_1^{(t)}\partial_2^{(s)}\partial_{12}v_2^1\|_0 
\leq C\tilde C_0^{1/2}\mu_0^{\gamma/2}\lambda_1^{t}\mu_0^{s+1}
+ C\tilde C_1^{1/2}\mu_1^{\gamma/2}\lambda_2^{t}\mu_1^{s+1}
\leq C\tilde C_1^{1/2}\mu_1^{\gamma/2}\lambda_2^{t}\mu_1^{s+1}, \\  
& \|\partial_1^{(t)}\partial_2^{(s)}\partial_{22}v_2^1\|_0
\leq C\tilde C_0^{1/2}\mu_0^{\gamma/2}\lambda_1^{t}\mu_0^{s+1}
+ C\tilde C_1^{1/2}\mu_1^{\gamma/2}\lambda_2^{t-1}\mu_1^{s+2}
\leq C\tilde C_1^{1/2}\mu_1^{\gamma/2}\lambda_2^{t-1}\mu_1^{s+2}, \vspace{2mm}\\
& \|\partial_1^{(t)}\partial_2^{(s)}\partial_{11}v_2^2\|_0 
\leq C\tilde C_0^{1/2}\mu_0^{\gamma/2}\lambda_1^{t+1}\mu_1^{s}
+ C\tilde C_1^{1/2}\mu_1^{\gamma/2}\lambda_2^{t+2}\mu_2^{s-1}
\leq C\tilde C_1^{1/2}\mu_1^{\gamma/2}\lambda_2^{t+2}\mu_2^{s-1}, \\  
& \|\partial_1^{(t)}\partial_2^{(s)}\partial_{12}v_2^2\|_0
\leq C\tilde C_0^{1/2}\mu_0^{\gamma/2}\lambda_1^{t+1}\mu_1^{s}
+ C\tilde C_1^{1/2}\mu_1^{\gamma/2}\lambda_2^{t+1}\mu_2^{s}
\leq C\tilde C_1^{1/2}\mu_1^{\gamma/2}\lambda_2^{t+1}\mu_2^{s},\\ 
& \|\partial_1^{(t)}\partial_2^{(s)}\partial_{22}v_2^2\|_0
\leq C\tilde C_0^{1/2}\mu_0^{\gamma/2}\lambda_1^{t}\mu_1^{s+1}
+ C\tilde C_1^{1/2}\mu_1^{\gamma/2}\lambda_2^{t}\mu_2^{s+1}
\leq C\tilde C_1^{1/2}\mu_1^{\gamma/2}\lambda_2^{t}\mu_2^{s+1}.
\end{split}
\end{equation}
Above, we used the following fact resulting from (\ref{two}):
$$\Big(\frac{\tilde C_0}{\tilde C_1}\Big)^{1/2}\leq \Big(\frac{\mu_1}{\lambda_1}\Big)^{1/2}
\leq \min\Big\{ \frac{\lambda_2}{\lambda_1}, \,
\frac{\mu_1}{\mu_0}, \, \frac{\mu_1^2}{\lambda_2\mu_0},\, \frac{\lambda_2^2}{\lambda_1\mu_2},\,
\frac{\mu_2}{\mu_1}\Big\}.$$
Indeed, comparison with $\lambda_2/\lambda_2$, $\mu_1/\mu_0$ and
$\mu_2/\mu_1$ is direct, whereas:
$$\frac{\mu_1^2}{\lambda_2\mu_0}= \frac{\mu_1}{\lambda_1} 
\frac{\lambda_1}{\mu_0} \frac{\mu_1}{\lambda_2} =
\frac{\mu_1}{\lambda_1} \geq \Big(\frac{\mu_1}{\lambda_1}\Big)^{1/2},
\qquad \frac{\lambda_2^2}{\lambda_1\mu_2} = 
\frac{\lambda_2}{\mu_1}\frac{\mu_1}{\lambda_1}\frac{\lambda_2}{\mu_2}
= \frac{\lambda_2}{\mu_1}\Big(\frac{\mu_1}{\lambda_1}\Big)^{1/2}
\geq \Big(\frac{\mu_1}{\lambda_1}\Big)^{1/2}. $$
Carrying out the remaining bounds as in steps 8 and 9 of the proof of
Proposition \ref{prop2} yields:
\begin{equation}\label{sec_w2}
\begin{split}
& \|w_2-w_1\|_1\leq C\tilde C_1^{1/2}\mu_1^{\gamma/2}\big(\|\nabla v_1\|_0
+ \tilde C_1^{1/2}\mu_1^{\gamma/2}\big)\\ & \qquad\qquad\quad \leq
C\tilde C_1^{1/2}\mu_1^{\gamma/2}\big(\|\nabla v_0\|_0 
+ \tilde C_0^{1/2}\mu_0^{\gamma/2}\big) \leq
C\tilde C_1^{1/2}\mu_1^{\gamma}\big(\|\nabla v\|_0 + \tilde C_0^{1/2}\big), 
\\ & \|\nabla^{2}(w_2-w_1)\|_0\leq 
C\tilde C_1^{1/2}\mu_1^{\gamma}\big(\|\nabla v\|_0 + \tilde C_0^{1/2}\big)\mu_2,
\end{split}
\end{equation}
in view of (\ref{a2}), (\ref{B1}), (\ref{sec_v1}) and (\ref{fir_v1}).

\smallskip

{\bf 6. (Iterations from $j$ to $j+1$ with $j=2\ldots K-1$)}
Assume that we carried out the construction of $v_j,
w_j\in\mathcal{C}^\infty(\bar\omega+\bar B_l(0), \R^2)$ and that,
after setting the following quantities: %in agreement with (\ref{DEF}):
$$\tilde C_j = C\frac{\tilde C_{j-1}}{(\mu_j/\lambda_j)^2},\qquad
B_j=C\tilde C_{j-1},\qquad \mu_j\sigma = \lambda_{j+1}<\mu_{j+1},$$
we get $\|\nabla^{(m)}\mathcal{D}_j\|_0\leq\tilde C_j\mu_j^m$ and
\ref{ass_HQ3}, valid for all $m, t+s=0\ldots (K-j)(3N+4)$. All these clearly holds at $j=2$, by (\ref{D2}) and
(\ref{sec_v2}). Observe that other assumptions of Proposition \ref{prop2}
are likewise valid, since $\lambda_{j+1}^{1-\gamma}\geq \mu_j\sigma_0$
in virtue of (\ref{Addi}), while (\ref{ass_ce2}) follows from:
$$\frac{B_j}{\tilde C_j} = \Big(\frac{\mu_j}{\lambda_j}\Big)^2 = 
\Big(\frac{\mu_{j+1}}{\lambda_{j+1}}\Big)^2 = \Big(\frac{\lambda_{j+1}}{\mu_j}\Big)^N = 
\sigma^N.$$
Consequently, we obtain $v_{j+1}, w_{j+1}\in\mathcal{C}^\infty(\bar\omega + \bar B_l(0), \R^2)$ such
that, in virtue of \ref{P2bound4}:
\begin{equation}\label{Dj-bound}
\|\nabla^{(m)}\mathcal{D}_{j+1}\|_0\leq C\mu_{j+1}^m\frac{\tilde C_j}{(\mu_{j+1}/\lambda_{j+1})^2} 
\quad\mbox{ for all } \; m=0\ldots (K-j+1)(3N+4).
\end{equation}
Also, \ref{P2bound2} implies that \ref{ass_HQ3} holds at $j+1$, for all $t+s=0\ldots (K-j+1)(3N+4),$ 
with $B_{j+1}=C\tilde C_j$, because the extra factors in the third and
fourth bounds there become:
\begin{equation*}
\frac{1}{(\lambda_j/\mu_{j-1})^2(\mu_j/\lambda_{j+1})}= \frac{1}{\sigma} \leq 1,\qquad 
\frac{1}{(\lambda_{j+1}/\lambda_j)(\lambda_{j+1}/\mu_{j+1})}= \frac{1}{\sigma^{N/2}} \leq 1,
\end{equation*}
once we have set the frequencies through (\ref{DEF}).
Further, \ref{P2bound1} and \ref{P2bound3} imply:
\begin{equation}\label{sec_w3}
\begin{split}
&\|v_{j+1}-v_j\|_1\leq C\tilde C_j^{1/2}\mu_j^{\gamma/2} \leq C\tilde C_0 \mu_j^{\gamma/2}, \\ &
\|w_{j+1}-w_j\|_1\leq C\tilde C_j^{1/2}\mu_j^\gamma\big(\|\nabla
v\|_0+\tilde C_0^{1/2}\big), \\ &
\|\nabla^2(w_{j+1}-w_j)\|_0\leq C\tilde C_j^{1/2}\mu_j^\gamma \big(\|\nabla
v\|_0+\tilde C_0^{1/2}\big)\mu_{j+1}.
\end{split}
\end{equation}

\smallskip

{\bf 7. (Proof of the final bounds)} After the total of $K$ steps, we declare:
$$\tilde v = v_K, \qquad \tilde w=w_K.$$
Using \ref{pr_stima1}, together with (\ref{fir_v1}), (\ref{sec_w1}),
(\ref{fir_v2}), (\ref{sec_w2}) and (\ref{sec_w3}) we arrive at the following version of the
estimates, claimed in \ref{Abound12}:
\begin{equation*}
\begin{split}
& \|\tilde v-v\|_1\leq \|v-v_0\|_1 +
\sum_{j=1}^{K-1}\|v_{j+1}-v_j\|_1\leq C\big(l\mathcal{M} + \tilde C_0^{1/2}\mu_{K}^{\gamma/2}\big)
\leq C\mu_K^{\gamma/2}(\|\mathcal{D}\|_0^{1/2} + l\mathcal{M}),
\\ & \|\tilde w-w\|_1\leq \|w-w_0\|_1 +
\sum_{j=1}^{K-1}\|w_{j+1}-w_j\|_1\leq C l\mathcal{M} + C\tilde
C_0^{1/2}\mu_{K}^{\gamma/2}(\|\nabla v\|_0+\tilde C_0^{1/2})
\\ & \qquad\qquad \, \leq C\mu_K^{\gamma/2}(\|\mathcal{D}\|_0^{1/2} + l\mathcal{M}) 
\big(1+\|\mathcal{D}\|_0^{1/2} + l\mathcal{M} + \|\nabla v\|_0\big).
\end{split}
\end{equation*}
Further, from (\ref{D1}), (\ref{D2}) and (\ref{Dj-bound}):
\begin{equation}\label{Cfin}
\|\mathcal{D}_j\|_0\leq \tilde C_j = C\frac{\tilde C_0}{\sigma^{jN}}\qquad\mbox{for all }\; j=0\ldots K,
\end{equation}
so that, recalling (\ref{DEF}) we get:
\begin{equation*}
\begin{split}
\tilde C_j^{1/2}\mu_{j+1} & = C \frac{\tilde
  C_0^{1/2}}{\sigma^{jN/2}}\sigma^{j+1+(j+2)N/2}\mu_0 = C \tilde
C_0^{1/2}\sigma^{j+1+N}\mu_0\\ & \leq C\tilde C_0^{1/2}\sigma^{K+N}\mu_0
\quad \mbox{ for all }\; j= \ldots K-1.
\end{split}
\end{equation*}
Therefore, recalling \ref{ass_HQ3} at $j+1=K$, we get a version of
the first bound in \ref{Abound22}:
\begin{equation*}
\begin{split}
\|\nabla^2 \tilde v\|_0 & \leq \mu_K^{\gamma/2} B_K^{1/2}\mu_K = C\tilde
C_{K-1}^{1/2}\mu_K^{\gamma/2}\mu_K 
\\ & = C \tilde C_0^{1/2} \mu_K^{\gamma/2}\sigma^{K+N}\mu_0
\leq C\mu_K^{\gamma/2}\frac{(\lambda l)^{K+N}}{l}(\|\mathcal{D}\|_0^{1/2} + l\mathcal{M}),
\end{split}
\end{equation*}
while the second bound follows through (\ref{sec_w1}), \ref{sec_w2}) and (\ref{sec_w3}):
\begin{equation*}
\begin{split}
\|\nabla^2 \tilde w\|_0 & \leq \|\nabla^2w_0\|_0 + \sum_{j=0}^{K-1}\|\nabla^2(w_{j+1}-w_j)\|_0
\leq Cl\mathcal{M}\mu_0 + \mu_K^{\gamma} \big(\|\nabla v\|_0 +
\tilde C_0^{1/2}\big) \sum_{j=0}^{K-1}\tilde C_{j}^{1/2}\mu_{j+1}
\\ & \leq C \tilde C_0^{1/2} \mu_K^{\gamma}\sigma^{K+N}\mu_0\big(
1+\|\mathcal{D}\|_0^{1/2} + l\mathcal{M} + \|\nabla v\|_0\big)
\\ & \leq C\mu_K^{\gamma}\frac{(\lambda l)^{K+N}}{l}(\|\mathcal{D}\|_0^{1/2} + l\mathcal{M})
\big(1+\|\mathcal{D}\|_0^{1/2} + l\mathcal{M} + \|\nabla v\|_0\big).
\end{split}
\end{equation*}
Finally, \ref{Abound32} results from \ref{pr_stima2} and (\ref{Cfin}):
\begin{equation*}
\begin{split}
\|\tilde{\mathcal{D}}\|_0 & \leq \|A-A_0\|_0 + \|\mathcal{D}_K\|_0 
\leq C l^{s+\beta}\|A\|_{s,\beta} + \tilde C_K \leq C \Big( l^{s+\beta}\|A\|_{s,\beta}+
\frac{\tilde C_0}{\sigma^{KN}}\Big) \\ & \leq
C\Big (l^{s+\beta}\|A\|_{s,\beta}+ \frac{\|\mathcal{D}\|_0 + (l\mathcal{M})^2}{(\lambda l)^{KN}}\Big). 
\end{split}
\end{equation*}

\smallskip

{\bf 8. (Reparametrizing $\gamma$)} 
The final bounds that we proved in step 7, carry the following factor in their right hand sides, in virtue of (\ref{DEF}):
\begin{equation*}
\begin{split}
& \mu_K^\gamma = \Big(\frac{\sigma^{K+(K+1)N/2}}{l}\Big)^\gamma=
\lambda^{K(1+N/2)\gamma+\gamma} l^{K(1+N/2)\gamma}\leq \lambda^{K(1+N/2)\gamma+\gamma}
=\lambda^{\bar\gamma},
\\ & \mbox{where }\; \bar\gamma = \big(K(1+N/2) + 1\Big)\gamma.
\end{split}
\end{equation*}
We also note that the second condition in
(\ref{Addi}) is equivalently written as:
$$\lambda^{1-\frac{\gamma}{1-K(1+N/2)\gamma+\gamma}} l \geq \sigma_0^{\frac{1}{1-K(1+N/2)\gamma+\gamma}},$$
which, using the notation of $\bar\gamma$, becomes:
\begin{equation}\label{neww}
\lambda^{1-\frac{\bar\gamma}{(K(1+N/2)+1)(1-\bar\gamma+2\gamma)}}l
\geq \sigma_0^{\frac{1}{1-\bar\gamma+2\gamma}}.
\end{equation}
The exponent on $\lambda$ in the left hand side above is greater than $1-\bar{\gamma}$, if we assume:
\begin{equation}\label{Addw}
\bar \gamma\leq \frac{1}{2}\qquad\mbox{ and }\qquad 
\lambda^{1-\bar\gamma}l\geq \sigma_0^2.
\end{equation}
We then validate (\ref{neww}) through (\ref{Addw}), because:
$$\lambda^{1-\frac{\bar\gamma}{(K(1+N/2)+1)(1-\bar\gamma+2\gamma)}}l
\geq \lambda^{1-\bar\gamma}l\geq \sigma_0^2\geq \sigma_0^{\frac{1}{1-\bar\gamma}}
\geq \sigma_0^{\frac{1}{1-\bar\gamma+2\gamma}}.$$
In summary, we treat $\lambda$ in the statement of Theorem
\ref{thm_stage} as the given $\bar\lambda$ which, without loss of
generality we decrease below $1/2$. We then define $\sigma_0$
to be the square of $\sigma_0$ from Proposition \ref{prop2}, whereas
the second assumption in (\ref{Assu}) is exactly the second condition
in (\ref{Addw}), implying the second condition in (\ref{Addi}). This
yields all the bounds in step 7. Replacing in their right hand side
the factor $\mu_K^\gamma$ by its majorant $\lambda^{\bar\gamma}$, we
obtain the claimed formulas \ref{Abound12}-\ref{Abound32}, with the
stated dependence of constants and parameters.
The proof is done.
\endproof

\section{The Nash-Kuiper scheme and a proof of Theorem \ref{th_final}}\label{sec4}

The proof of Theorem \ref{th_final} relies on iterating Theorem \ref{thm_stage} using the Nash-Kuiper scheme. We quote the slightly altered version of the main recursion argument in \cite{lew_improved, lew_improved2}, which is similar to  \cite[section 6]{CDS}, now involving the H\"older norms, necessary in view of Lemma \ref{lem_diagonal}. The proof, allowing now for $s\neq 0$ follows by inspection of the original calculations in \cite{lew_improved}.

\begin{theorem}\label{th_NK}\cite[Theorem 1.4]{lew_improved}
  \cite[Lemma 5.2]{lew_improved2}
Let $\omega\subset\R^d$ be an open, bounded and smooth domain,
and let $k, J, S\geq 1$. Assume that there exists $l_0\in (0,1)$ such that
the following holds for every $l\in (0, l_0]$. Given
the fields $v\in\mathcal{C}^2(\bar\omega+\bar B_{2l}(0), \R^k)$, $w\in\mathcal{C}^2(\bar\omega+\bar B_{2l}(0), \R^d)$, $A\in\mathcal{C}^{s,\beta}(\bar\omega+\bar B_{2l}(0), \R^{d\times d}_\sym)$ where $s\in\{0,1\}$, $\beta\in (0,1]$ with $s+\beta<2$, and given $\gamma, \lambda, \mathcal{M}$ which satisfy,
together with $\sigma_0\geq 1$ that depends on $\omega, k, S, J,\gamma$, the following bounds: 
% (but not on $l,\lambda, v,w$):
\begin{equation}\label{ass_impro2}
\gamma\in (0,1),\qquad \lambda^{1-\gamma} l>\sigma_0,
\qquad \mathcal{M}\geq \max\{\|v\|_2, \|w\|_2, 1\},
\end{equation}
there exist  $\tilde v\in\mathcal{C}^2(\bar \omega+\bar B_l(0),\R^k)$,
$\tilde w\in\mathcal{C}^2(\bar\omega+\bar B_l(0),\R^d)$ satisfying:
%\begin{equation}\label{stage_est}
\begin{align*}
& \hspace{-3mm} \left. \begin{array}{l} \|\tilde v - v\|_1\leq
C\lambda^{\gamma/2}\big(\|\mathcal{D}\|_0^{1/2}+l\mathcal{M}\big), \vspace{1mm} \\ 
\|\tilde w - w\|_1\leq C\lambda^{\gamma}\big(\|\mathcal{D}\|_0^{1/2}+l\mathcal{M}\big)
\big(1+ \|\mathcal{D}\|_0^{1/2}+l\mathcal{M}+\|\nabla v\|_0\big), \end{array}\right.%\}
\vspace{5mm}\\
& \hspace{-3mm} \left. \begin{array}{l} \|\nabla^2\tilde v\|_0\leq C{\displaystyle{\frac{(\lambda
 l)^J}{l}\lambda^{\gamma/2}}}\big(\|\mathcal{D}\|_0^{1/2}+l\mathcal{M}\big),\vspace{1mm}\\ 
\|\nabla^2\tilde w\|_0\leq C{\displaystyle{\frac{(\lambda
  l)^J}{l}}}\lambda^{\gamma}\big(\|\mathcal{D}\|_0^{1/2}+l\mathcal{M}\big)
\big(1+\|\mathcal{D}\|_0^{1/2}+l\mathcal{M}+\|\nabla v\|_0\big), \end{array}\right.
\medskip\\ 
& \|\tilde{\mathcal{D}}\|_0\leq C\Big(l^{s+\beta}{\|A\|_{s,\beta}}
+\frac{\lambda^{\gamma}}{(\lambda l)^S} \big(
\|\mathcal{D}\|_0 + (l\mathcal{M})^2\big)\Big),
\end{align*}
%\end{equation}
with constants $C$ that depend only on $\omega, k, J,S,\gamma$,
and where we denote the defects:
$$\mathcal{D}=A -\big(\frac{1}{2}(\nabla v)^T\nabla v + \sym\nabla
w\big),\qquad \tilde{\mathcal{D}}=A -\big(\frac{1}{2}(\nabla \tilde
v)^T\nabla \tilde v + \sym\nabla \tilde w\big).$$
Then, for every triple of fields $v, w, A$ as above, which additionally
satisfy the defect smallness condition $0<\|\mathcal{D}\|_0\leq 1$, 
and for every exponent $\alpha$ in the range:
\begin{equation}\label{rangeAlz}
0< \alpha <\min\Big\{\frac{s+\beta}{2},\frac{S}{S+2J}\Big\},
\end{equation}
there exist $\bar v\in\mathcal{C}^{1,\alpha}(\bar\omega,\R^k)$ and
$\bar w\in\mathcal{C}^{1,\alpha}(\bar\omega,\R^d)$ with the following properties:
\begin{align*}
& \|\bar v - v\|_1\leq C \big(1+\|\nabla v\|_0\big)^2
\|\mathcal{D}_0\|_0^{1/4}, \quad \|\bar w -
w\|_1\leq C(1+\|\nabla v\|_0)^3\|\mathcal{D}\|_0^{1/4}, \vspace{1mm}\\
& A-\big(\frac{1}{2}(\nabla \bar v)^T\nabla \bar v + \sym\nabla
\bar w\big) =0 \quad\mbox{ in }\; \bar\omega. 
\end{align*}
The constants $C$ above depend only on $\omega, k, A$ (and hence on $s,\beta$) and $\alpha$. 
\end{theorem}

\smallskip

Clearly, Theorem \ref{th_NK} and Theorem \ref{thm_stage} yield together
the following result below, because:
\begin{equation*}
\begin{split}
\frac{S}{S+2J} = \frac{KN}{KN+2(K+N)} & \to
\frac{N}{N+2} \quad \mbox{ as } \; K\to\infty \\
& \to 1 \qquad \quad \; \,\mbox{ as } \;N\to\infty.
\end{split}
\end{equation*}

\begin{corollary}\label{th_NKH}
Let $\omega\subset\R^2$ be an open, bounded and smooth domain.
Fix $\alpha$ as in (\ref{VKrange}).
Then, there exists $l_0\in (0,1)$ such that, for every $l\in (0, l_0]$
and for every $v, w\in\mathcal{C}^2(\bar\omega + \bar B_{2l}(0),\R^2)$,
$A\in\mathcal{C}^{s,\beta}(\bar\omega +\bar B_{2l}(0), \R^{2\times 2}_\sym)$ where $s\in\{0,1\}$, $\beta\in (0,1]$ with $s+\beta<2$, and such that:
$$\mathcal{D}=A-\big(\frac{1}{2}(\nabla v)^T\nabla v + \sym\nabla
w\big) \quad\mbox{ satisfies } \quad 0<\|\mathcal{D}\|_0\leq 1,$$
there exist $\tilde v, \tilde w \in\mathcal{C}^{1,\alpha}(\bar\omega,\R^2)$
with the following properties:
\begin{align*}
& \|\tilde v - v\|_1\leq C(1+\|\nabla v\|_0)^2\|\mathcal{D}\|_0^{1/4}, \quad \|\tilde w -
w\|_1\leq C (1+\|\nabla v\|_0)^3\|\mathcal{D}\|_0^{1/4}, \vspace{1mm}\\
& A-\big(\frac{1}{2}(\nabla \tilde v)^T\nabla \tilde v + \sym\nabla
\tilde w\big) =0 \quad\mbox{ in }\; \bar\omega. 
\end{align*}
%\end{equation}
The norms in
the left hand side above are taken on $\bar\omega$, and in the right hand
side on $\bar\omega+ \bar B_{2l}(0)$. The constants $C$ depend only
on $\omega, A$ (and hence on $s,\beta$) and $\alpha$. 
\end{corollary}

\bigskip

\noindent The proof of Theorem \ref{th_final} is consequently the same as the proof of
Theorem 1.1 in \cite{lew_improved}, in section 5 there. We 
replace $\omega$ by its smooth superset, and apply the basic stage
construction in order to first decrease $\|\mathcal{D}\|_0$
below $1$. Then, Corollary \ref{th_NKH} yields the final result. \endproof

%\bibliographystyle{plain}
%\bibliography{bib}

\end{document}